\documentclass[11pt]{amsart}
\usepackage{amssymb}
\usepackage{amsthm}
\usepackage{amsmath}
\usepackage[dvips]{epsfig}
\usepackage{graphicx}
\usepackage{color}
\usepackage{url}
\usepackage{tikz}
\usepackage{caption}
\usepackage{subcaption}
\usepackage{epstopdf}
\usepackage[arrow, matrix, curve]{xy}
\usepackage{framed}
\usetikzlibrary{decorations.markings}

\makeatletter\@addtoreset {equation}{section}\makeatother

\usetikzlibrary{matrix,arrows}
\newtheorem{thm}{Theorem}

\newtheorem{lem}{Lemma}
\newtheorem{mydef}{Definition}
\newtheorem{remark}{Remark}
\newtheorem{assumption}{Assumption}
\newtheorem{rhp}{Riemann-Hilbert problem}
{\begin{trivlist} \item[]{\em Proof }}%
{\hspace*{\fill}$\rule{.3\baselineskip}{.35\baselineskip}$\end{trivlist}}

\DeclareMathOperator{\Imag}{Im}
\DeclareMathOperator{\Real}{Re}

\DeclareMathOperator{\C}{\mathbb{C}}

\DeclareMathOperator{\R}{\mathbb{R}}

\DeclareMathOperator{\eps}{\varepsilon}

\begin{document}

\title[Inverse scattering for the massive Thirring model]{Inverse scattering for the massive Thirring model}

\author{Dmitry E. Pelinovsky}
\address[D.E. Pelinovsky]{Department of Mathematics, McMaster University, Hamilton ON, Canada, L8S 4K1}
\email{dmpeli@math.mcmaster.ca}

\author{Aaron Saalmann}
\address[A. Saalmann]{Mathematisches Institut, Universit\"{a}t zu K\"{o}ln, 50931 K\"{o}ln, Germany}
\email{asaalman@math.uni-koeln.de}

\thanks{A.S. gratefully acknowledges financial support from the project SFB-TRR 191
``Symplectic Structures in Geometry, Algebra and Dynamics" (Cologne University, Germany).}

\begin{abstract}
We consider the massive Thirring model in the laboratory coordinates and explain how the inverse scattering transform
can be developed with the Riemann--Hilbert approach. The key ingredient of our technique is to transform the corresponding
spectral problem to two equivalent forms: one is suitable for the spectral
parameter at the origin and the other one is suitable for the spectral parameter at infinity. Global solutions to the massive Thirring model
are recovered from the reconstruction formulae at the origin and at infinity.
\end{abstract}

\maketitle

\section{Introduction}

The massive Thirring model (MTM) was derived by Thirring in 1958 \cite{Thirring}
in the context of general relativity. It represents a relativistically invariant
nonlinear Dirac equation in the space of one dimension. Another relativistically
invariant one-dimensional Dirac equation is given by the Gross--Neveu model \cite{gross-neveu}
also known as the massive Soler model \cite{Soler} when it is written in the space of three dimensions.

It was discovered in 1970s by Mikhailov \cite{M}, Kuznetsov and Mikhailov \cite{KM},
Orfanidis \cite{O}, Kaup and Newell \cite{KN}
that the MTM is integrable with the inverse scattering transform method
in the sense that it admits a Lax pair, countably many conserved quantities, the B\"{a}cklund
transformation, and other common features of integrable models. We write the MTM system
in the laboratory coordinates by using the normalized form:
\begin{equation}\label{e mtm}
  \left\{
    \begin{array}{ll}
      i(u_t+u_x) + v + |v|^2u = 0,\\
      i(v_t-v_x) + u + |u|^2v = 0.
    \end{array}
  \right.
\end{equation}
The MTM system (\ref{e mtm}) appears as the compatibility condition in the Lax representation
\begin{equation}\label{e mtm presentation}
  L_t-A_x+[L,A]=0,
\end{equation}
where the $2\times 2$-matrices $L$ and $A$ are given by
\begin{equation}\label{e def L}
  L=\frac{i}{4}(|u|^2-|v|^2)\sigma_3-
  \frac{i\lambda}{2}
  \left(
    \begin{array}{cc}
      0 & \overline{v} \\
      v & 0 \\
    \end{array}
  \right)+
  \frac{i}{2\lambda}
  \left(
    \begin{array}{cc}
      0 & \overline{u} \\
      u & 0 \\
    \end{array}
  \right)+
  \frac{i}{4}\left(\lambda^2-\frac{1}{\lambda^2}\right) \sigma_3
\end{equation}
and
\begin{equation}\label{e def A}
  A=-\frac{i}{4}(|u|^2+|v|^2)\sigma_3-
  \frac{i\lambda}{2}
  \left(
    \begin{array}{cc}
      0 & \overline{v} \\
      v & 0 \\
    \end{array}
  \right)-
  \frac{i}{2\lambda}
  \left(
    \begin{array}{cc}
      0 & \overline{u} \\
      u & 0 \\
    \end{array}
  \right)+
  \frac{i}{4}\left(\lambda^2+\frac{1}{\lambda^2}\right) \sigma_3.
\end{equation}
Other forms of $L$ and $A$ with nonzero trace have also been introduced by Barashenkov and Getmanov \cite{BG87}.
The traceless representation of $L$ and $A$ in (\ref{e def L}) and (\ref{e def A}) is more useful for inverse scattering transforms.

Formal inverse scattering results for the linear operators
 (\ref{e def L}) and (\ref{e def A}) can be found in \cite{KM}.
The first steps towards rigorous developments of the inverse scattering transform for the MTM system (\ref{e mtm})
were made in 1990s by Villarroel \cite{V} and Zhou \cite{Zhou}. In the former work, the treatment of
the Riemann--Hilbert problems is sketchy, whereas in the latter work, an abstract method to solve Riemann--Hilbert problems
with rational spectral dependence is developed with applications to the sine--Gordon equation in the laboratory
coordinates. Although the MTM system (\ref{e mtm}) does not appear in the list of examples in \cite{Zhou},
one can show that the abstract method of Zhou is also applicable to the MTM system.

The present paper relies on recent progress in the inverse scattering transform method
for the derivative NLS equation \cite{PSS,Pelinovsky-Shimabukuro}. The key element of our technique
is a transformation of the spectral plane $\lambda$ for the operator $L$ in (\ref{e def L}) to
the spectral plane $z = \lambda^2$ for a different spectral problem. This transformation can be
performed uniformly in the $\lambda$ plane for the Kaup--Newell spectral
problem related to the derivative NLS equation \cite{KN78}. In the contrast, one needs to
consider separately the subsets of the $\lambda$ plane near the origin and near infinity
for the operator $L$ in (\ref{e def L}) due to its rational dependence on $\lambda$.
Therefore, two Riemann--Hilbert problems are derived for the MTM system (\ref{e mtm}) with
the components $(u,v)$: the one near $\lambda = 0$ recovers $u$ and the other one near $\lambda = \infty$
recovers $v$.

Let $\dot{L}^{2,m}(\mathbb{R})$ denote the space of square integrable functions with the weight $|x|^m$
for $m \in \mathbb{Z}$ so that $L^{2,m}(\mathbb{R}) \equiv \dot{L}^{2,m}(\R) \cap L^2(\R)$.
Let $\dot{H}^{n,m}(\mathbb{R})$ denote the Sobolev space of functions, the $n$-th derivative of which is square integrable
with the weight $|x|^m$ for $n \in \mathbb{N}$ and $m \in \mathbb{Z}$
so that $H^{n,m}(\mathbb{R}) \equiv \dot{H}^{n,m}(\R) \cap \dot{L}^{2,m}(\R) \cap H^n(\R)$
with $H^n(\R) \equiv \dot{H}^n(\R) \cap L^2(\R)$. Norms on any of these spaces are
introduced according to the standard convention.

The inverse scattering transform for the linear operators (\ref{e def L}) and (\ref{e def A})
can be controlled when the potential $(u,v)$ belongs to the function space
\begin{equation}
\label{X-2-1}
X_{(u,v)} := H^2(\mathbb{R}) \cap H^{1,1}(\mathbb{R}).
\end{equation}
Transformations of the spectral plane employed here allow us to give a sharp requirement
on the $L^2$-based Hilbert spaces, for which the Riemann--Hilbert problem can be solved
by using the technique from Deift and Zhou \cite{D-Z-1,Z}. Note that both the direct and inverse
scattering transforms for the NLS equation are solved in function space $H^1(\mathbb{R}) \cap L^{2,1}(\mathbb{R})$,
which is denoted by the same symbol $H^{1,1}(\R)$ in the previous works \cite{D-Z-1,Z}.
Compared to this space, the reflection coefficients $(r_+,r_-)$ introduced
in our paper for the linear operators (\ref{e def L}) and (\ref{e def A}) belong to the function space
\begin{equation}
\label{X-1-2}
X_{(r_+,r_-)} := \dot{H}^1(\R \setminus [-1,1]) \cap \dot{H}^{1,1}([-1,1])
\cap \dot{L}^{2,1}(\mathbb{R}) \cap \dot{L}^{2,-2}(\mathbb{R}).
\end{equation}

In the application of the inverse scattering transform to the derivative NLS equation,
alternative methods were recently developed \cite{Perry1,Perry} based on a different
(gauge) transformation of the Kaup--Newell spectral problem to the spectral problem for the
Gerdjikov--Ivanov equation. Both the potentials and the reflection coefficients
were controlled in the same function space $H^2(\mathbb{R}) \cap L^{2,2}(\mathbb{R})$
\cite{Perry1,Perry}. These function spaces are more restrictive compared to the function spaces
for the potential and the reflection coefficients used in \cite{PSS,Pelinovsky-Shimabukuro}.

Unlike the recent literature on the derivative NLS equation, our interest to the inverse scattering for the
MTM system (\ref{e mtm}) is not related to the well-posedness problems. Indeed, the local and global existence
of solutions to the Cauchy problem for the MTM system (\ref{e mtm}) in the $L^2$-based Sobolev spaces
$H^m(\mathbb{R})$, $m \in \mathbb{N}$ can be proven with the standard contraction and energy methods,
see review of literature in \cite{Pel-survey}. Low regularity solutions in $L^2(\mathbb{R})$ were already
obtained for the MTM system by Selberg and Tesfahun \cite{SelTes}, Candy \cite{Candy}, Huh \cite{Huh11,Huh13,Huh15}, and Zhang \cite{Zhang,ZhangZhao}.
The well-posedness results can be formulated as follows.

\begin{thm}\cite{Candy,Huh15}
For every $(u_0,v_0) \in H^m(\mathbb{R})$, $m \in \mathbb{N}$, there exists a unique global solution $(u,v) \in C(\mathbb{R},H^m(\R))$
such that $(u,v) |_{t = 0} = (u_0,v_0)$ and the solution $(u,v)$ depends continuously on the initial data $(u_0,v_0)$.
Moreover, for every $(u_0,v_0) \in L^2(\mathbb{R})$, there exists a global solution $(u,v) \in C(\mathbb{R},L^2(\R))$
such that $(u,v) |_{t = 0} = (u_0,v_0)$. The solution $(u,v)$ is unique in a certain subspace of $C(\R,L^2(\R))$ and
it depends continuously on the initial data $(u_0,v_0)$.
\label{theorem-well-posed}
\end{thm}

The inverse scattering transform and the reconstruction formulas for the global solutions
$(u,v)$ to the MTM system (\ref{e mtm}) can be used to solve other interesting analytical
problems such as long-range scattering to zero \cite{CandyLindblad},
orbital and asymptotic stability of the Dirac solitons \cite{CPS,PS2}, and an analytical proof of
the soliton resolution conjecture. Similar questions have been recently addressed
in the context of the cubic NLS equation \cite{CP14,C-P,S2017} and
the derivative NLS equation \cite{JLPS,LPS}.

The goal of our paper is to explain how the inverse scattering transform
for the linear operators (\ref{e def L}) and
(\ref{e def A}) can be developed by using the Riemann--Hilbert problem.
For simplicity of presentation, we assume that the initial data to the MTM system (\ref{e mtm})
admit no eigenvalues and resonances in the sense of Definition \ref{def-eigenvalues} given in Section \ref{sec-coef}.
Note that eigenvalues can be easily added by using B\"{a}cklund transformation
for the MTM system \cite{CPS}, whereas resonances can be removed by perturbations
of initial data \cite{Beals1984} (see relevant results in \cite{PSS}).
The following theorem represents the main result of this paper.

\begin{thm}
\label{theorem-main}
For every $(u_0,v_0) \in X_{(u,v)}$ admitting no eigenvalues or resonances in the sense
of Definition \ref{def-eigenvalues}, there is a direct scattering transform with the spectral data
$(r_+,r_-)$ defined in $X_{(r_+,r_-)}$. The unique solution $(u,v) \in C(\mathbb{R},X_{(u,v)})$ to the
MTM system (\ref{e mtm}) can be uniquely recovered by means of the inverse scattering transform
for every $t \in \mathbb{R}$.
\end{thm}

The paper is organized as follows. Section \ref{sec-Jost} describes Jost functions obtained
after two transformations of the differential operator $L$ given by (\ref{e def L}). Section \ref{sec-coef}
is used to set up scattering coefficients $(r_+,r_-)$ and to introduce the scattering relations between
the Jost functions. Section \ref{sec-RH} explains how the Riemann--Hilbert problems
can be solved and how the potentials $(u,v)$ can be recovered
in the time evolution of the MTM system (\ref{e mtm}). Section \ref{sec-conclusion}
concludes the paper with a review of open questions.

\section{Jost functions}
\label{sec-Jost}

The linear operator $L$ in (\ref{e def L}) can be rewritten in the form:
\begin{equation*}
  L = Q(\lambda;u,v)+\frac{i}{4}\left(\lambda^2-\frac{1}{\lambda^2}\right) \sigma_3,
\end{equation*}
where
$$
  Q(\lambda;u,v) = \frac{i}{4}(|u|^2-|v|^2)\sigma_3-
  \frac{i\lambda}{2}
  \left(
    \begin{array}{cc}
      0 & \overline{v} \\
      v & 0 \\
    \end{array}
  \right)+
  \frac{i}{2\lambda}
  \left(
    \begin{array}{cc}
      0 & \overline{u} \\
      u & 0 \\
    \end{array}
  \right).
$$
Here we freeze the time variable $t$ and drop it from the list of arguments.
Assuming fast decay of $(u,v)$ to $(0,0)$ as $|x| \to \infty$,
solutions to the spectral problem
\begin{equation}
\label{L-spectral}
\psi_x = L \psi
\end{equation}
can be defined by the following asymptotic behavior:
\begin{eqnarray*}
  \psi^{(-)}_1(x;\lambda)\sim
  \left(
    \begin{array}{c}
      1 \\
      0 \\
    \end{array}
  \right)e^{i x(\lambda^2-\lambda^{-2})/4}
  ,\qquad
  \psi^{(-)}_2(x;\lambda)\sim
  \left(
    \begin{array}{c}
      0 \\
      1 \\
    \end{array}
  \right)e^{-i x(\lambda^2-\lambda^{-2})/4}
  \quad \text{ as } x\to -\infty
\end{eqnarray*}
and
\begin{eqnarray*}
  \psi^{(+)}_1(x;\lambda)\sim
  \left(
    \begin{array}{c}
      1 \\
      0 \\
    \end{array}
  \right)e^{i x(\lambda^2-\lambda^{-2})/4}
  ,\qquad
  \psi^{(+)}_2(x;\lambda)\sim
  \left(
    \begin{array}{c}
      0 \\
      1 \\
    \end{array}
  \right)e^{-i x(\lambda^2-\lambda^{-2})/4}
  \quad \text{ as } x\to+\infty.
\end{eqnarray*}

The \emph{normalized Jost functions}
\begin{equation}
\label{norm-Jost}
    \varphi_{\pm}(x;\lambda)=\psi^{(\pm)}_1(x;\lambda) e^{-i x(\lambda^2-\lambda^{-2})/4},\qquad
    \phi_{\pm}(x;\lambda)=\psi^{(\pm)}_2(x;\lambda) e^{i x(\lambda^2-\lambda^{-2})/4}
\end{equation}
satisfy the constant boundary conditions at infinity:
\begin{equation}\label{e asymptotics psi}
    \lim_{x\to\pm\infty}\varphi_{\pm}(x;\lambda)=e_1\quad\text{ and } \quad\lim_{x\to\pm\infty}\phi_{\pm}(x;\lambda)=e_2,
\end{equation}
where $e_1=(1,0)^T$ and $e_2=(0,1)^T$. The normalized Jost functions are
solutions to the following Volterra integral equations:
\begin{subequations}\label{e volterra phi varphi}
\begin{eqnarray}\label{e volterra phi}
        \varphi_{\pm}(x;\lambda)&=e_1 + \int_{\pm\infty}^x
        \left(
          \begin{array}{cc}
            1 & 0 \\
            0 & e^{-\frac{i}{2}(\lambda^2-\lambda^{-2})(x-y)} \\
          \end{array}
        \right)
        Q(\lambda;u(y),v(y))\;\varphi_{\pm}(y;\lambda)dy,\\
        \label{e volterra varphi}
        \phi_{\pm}(x;\lambda)&=e_2 + \int_{\pm\infty}^x
        \left(
          \begin{array}{cc}
            e^{\frac{i}{2}(\lambda^2-\lambda^{-2})(x-y)} & 0 \\
            0 & 1\\
          \end{array}
        \right)
        Q(\lambda;u(y),v(y))\;\phi_{\pm}(y;\lambda)dy.
    \end{eqnarray}
\end{subequations}

A standard assumption in analyzing Volterra integral equations is $Q(\lambda;u(\cdot),v(\cdot))\in L^1(\R)$
for fixed $\lambda\neq 0$ which is equivalent to $(u,v)\in L^1(\R)\cap L^2(\R)$ by the definition of $Q$.
In this case, for every $\lambda \in (\R\cup i\R)\setminus\{0\}$, Volterra integral
equations (\ref{e volterra phi varphi}) admit unique solutions $\varphi_{\pm}(\cdot;\lambda)$ and $\phi_{\pm}(\cdot;\lambda)$ in the space $L^{\infty}(\R)$. However, even if $(u,v) \in L^1(\R)\cap L^2(\R)$ the $L^1$-norm of $Q(\lambda;u(\cdot),v(\cdot))$ is not controlled uniformly in $\lambda$ as $\lambda \to 0$ and $|\lambda| \to \infty$. This causes difficulties in studying the behaviour of $\varphi_{\pm}(\cdot;\lambda)$ and $\phi_{\pm}(\cdot;\lambda)$
as $\lambda \to 0$ and $|\lambda| \to \infty$ and thus we need to transform the spectral problem (\ref{L-spectral})
to two equivalent forms. These two transformations generalize the exact transformation
of the Kaup--Newell spectral problem to the Zakharov--Shabat spectral problem, see \cite{KN78,Pelinovsky-Shimabukuro}.

\subsection{Transformation of the Jost functions for small $\lambda$}

Assume $u \in L^{\infty}(\R)$, $\lambda \neq 0$, and define the transformation matrix by
\begin{equation}\label{e def T}
  T(u;\lambda):=
  \left(
    \begin{array}{cc}
      1 & 0 \\
      u & \lambda^{-1}
    \end{array}
  \right).
\end{equation}
Let $\psi$ be a solution of the spectral problem (\ref{L-spectral}) and define $\Psi := T \psi$.
Straightforward computations show that $\Psi$ satisfies the equivalent linear equation
\begin{equation}\label{e Lax1 Psi}
  \Psi_x=\mathcal{L}\Psi,
\end{equation}
with new linear operator
\begin{equation}
\label{L-expansion1}
   \mathcal{L} = Q_1(u,v)+\lambda^2 Q_2(u,v)+ \frac{i}{4}\left(\lambda^2- \frac{1}{\lambda^2}\right)\sigma_3
\end{equation}
where
\begin{equation*}
  Q_1(u,v) =
  \left(
    \begin{array}{cc}
      -\frac{i}{4}(|u|^2+|v|^2) & \frac{i}{2}\overline{u} \\
      u_x-\frac{i}{2}u|v|^2- \frac{i}{2}v & \frac{i}{4}(|u|^2+|v|^2) \\
    \end{array}
  \right), \quad
  Q_2 (u,v) = \frac{i}{2}
  \left(
    \begin{array}{cc}
     u\overline{v} & -\overline{v} \\
      u+u^2\overline{v} & -u\overline{v} \\
    \end{array}
  \right).
\end{equation*}

Let us define $z := \lambda^2$ and introduce the partition $\mathbb{C} = B_0 \cup \mathbb{S}^1 \cup B_{\infty}$ with
\begin{equation}\label{e def B_1}
  B_0 := \{z\in\C: \quad |z|<1\}, \quad \mathbb{S}^1 := \{z\in\C: \quad |z|=1\}, \quad
  B_{\infty} := \{ z \in \C : \quad |z| > 1 \}.
\end{equation}
The second term in (\ref{L-expansion1}) is bounded if $z \in B_0$.
The normalized Jost functions associated to the spectral problem (\ref{e Lax1 Psi})
denoted by $\{m_{\pm},n_{\pm}\}$ can be obtained from the original Jost functions
$\{ \varphi_{\pm},\psi_{\pm}\}$ by the transformation formulas:
\begin{equation}\label{e m n trafo}
  m_{\pm}(x;z) = T(u(x);\lambda)\varphi_{\pm}(x;\lambda),\qquad
  n_{\pm}(x;z) = \lambda\, T(u(x);\lambda)\phi_{\pm}(x;\lambda),
\end{equation}
subject to the constant boundary conditions at infinity:
\begin{equation}\label{e asymptotics psi m n}
    \lim_{x\to\pm\infty} m_{\pm}(x;\lambda)=e_1\quad\text{ and } \quad\lim_{x\to\pm\infty} n_{\pm}(x;\lambda)=e_2.
\end{equation}
The transformed Jost functions are
solutions to the following Volterra integral equations:
\begin{subequations}\label{e volterra m n}
\begin{eqnarray}\label{e volterra m}
 & \phantom{t} &  m_{\pm}(x;z) = e_1 \\
 \nonumber
 & \phantom{t} & \phantom{text} \phantom{text} +\!\! \int_{\pm\infty}^x
        \left(
          \begin{array}{cc}
            1 & 0 \\
            0 & e^{-\frac{i}{2}(z-z^{-1})(x-y)} \\
          \end{array}
        \right)
        \left[ Q_1(u(y),v(y)) + z Q_2(u(y),v(y))\right] m_{\pm}(y;z)dy,\\
        \label{e volterra n}
& \phantom{t} &       n_{\pm}(x;z) =  \\
 \nonumber
 & \phantom{t} & \phantom{text} \phantom{text} e_2 + \!\!\int_{\pm\infty}^x
        \left(
          \begin{array}{cc}
            e^{\frac{i}{2}(z-z^{-1})(x-y)} & 0 \\
            0 & 1\\
          \end{array}
        \right)
        \left[ Q_1(u(y),v(y)) + z Q_2(u(y),v(y))\right] n_{\pm}(y;z)dy.
    \end{eqnarray}
\end{subequations}
Compared to \cite{Pelinovsky-Shimabukuro}, we have an additional term $\frac{i}{2} z (x-y)$
in the argument of the oscillatory kernel and the additional term $z Q_2(u,v)$ under the integration sign. However,
both additional terms are bounded in $B_0$ where $|z| < 1$. Therefore,
the same analysis as in the proof of Lemmas 1 and 2 in \cite{Pelinovsky-Shimabukuro} yields the following.

\begin{lem}\label{l solvability m n}
  Let $(u,v)\in L^1(\R)\cap L^{\infty}(\R)$ and $u_x\in L^1(\R)$.
  For every $z \in (-1,1)$, there exist unique solutions $m_{\pm}(\cdot;z)\in L^{\infty}(\R)$ and
  $n_{\pm}(\cdot;z)\in L^{\infty}(\R)$ satisfying the integral equations (\ref{e volterra m n}).
  For every $x\in\R$, $m_{\pm}(x,\cdot)$ and $n_{\mp}(x,\cdot)$ are continued analytically in
  $\C^{\pm} \cap B_0$.  There exist a positive constant $C$ such that
  \begin{equation}\label{e global bound m n}
    \|m_{\pm}(\cdot;z)\|_{L^{\infty}}+ \|n_{\mp}(\cdot;z)\|_{L^{\infty}}\leq C, \quad z\in \C^{\pm} \cap B_0.
  \end{equation}
\end{lem}

\begin{lem}\label{l expansion m n}
  Under the conditions of Lemma \ref{l solvability m n}, for every $x\in\R$
  the normalized Jost functions $m_{\pm}$ and $n_{\pm}$ satisfy
  the following limits as $\Imag(z) \to 0$ along a contour in the domains of their analyticity:
  \begin{equation}\label{e expansion m n}
    \lim_{z \to 0} \frac{m_{\pm}(x;z)}{m_{\pm}^{\infty}(x)} = e_1, \quad
    \lim_{z \to 0} \frac{n_{\pm}(x;z)}{n_{\pm}^{\infty}(x)} = e_2,
   \end{equation}
where
  \begin{equation*}
  m_{\pm}^{\infty}(x)= e^{-\frac{i}{4}\int_{\pm\infty}^{x}(|u|^2+|v|^2)dy}, \quad
  n_{\pm}^{\infty}(x)= e^{\frac{i}{4}\int_{\pm\infty}^{x}(|u|^2+|v|^2)dy}.
  \end{equation*}
If in addition $u \in C^1(\mathbb{R})$, then
\begin{subequations} \label{e expansion m n-longer}
\begin{eqnarray} \label{e expansion m-longer}
\lim_{z \to 0} \frac{1}{z} \left[ \frac{m_{\pm}(x;z)}{m_{\pm}^{\infty}(x)} - e_1 \right] & = &
      \left(
        \begin{array}{c}
          -\int_{\pm\infty}^x \left[ \overline{u} (u_x-\frac{i}{2}u|v|^2- \frac{i}{2}v)-\frac{i}{2} u\overline{v}\right]dy \\
          2i u_x + u |v|^2 + v
        \end{array}
      \right),\\
\label{e expansion n-longer}
    \lim_{z \to 0} \frac{1}{z} \left[ \frac{n_{\pm}(x;z)}{n_{\pm}^{\infty}(x)} - e_2 \right] & = &
      \left(
        \begin{array}{c}
          \overline{u}\\
          \int_{\pm\infty}^x \left[\overline{u}(u_x-\frac{i}{2}u|v|^2- \frac{i}{2}v)-\frac{i}{2}u\overline{v}\right] dy \\
        \end{array}
      \right).
    \end{eqnarray}
\end{subequations}
\end{lem}

\begin{remark}
By Sobolev's embedding of $H^1(\mathbb{R})$ into the space of continuous, bounded, and decaying at infinity functions,
if $u \in H^1(\mathbb{R})$, then $u \in C(\mathbb{R}) \cap L^{\infty}(\mathbb{R})$ and $u(x) \to 0$ as $|x| \to \infty$.
By the embedding of $L^{2,1}(\R)$ into $L^1(\R)$, if $u \in H^{1,1}(\R)$, then $u \in L^1(\R)$ and $u_x\in L^1(\R)$.
Thus, requirements of Lemma \ref{l solvability m n} are satisfied
if $(u,v) \in H^{1,1}(\R)$. The additional requirement
$u \in C^1(\R)$ of Lemma \ref{l expansion m n} is satisfied if $u \in H^2(\R)$.
Hence, $X_{(u,v)}$ in (\ref{X-2-1})  is an optimal $L^2$-based Sobolev space for direct scattering
of the MTM system (\ref{e mtm}).
\end{remark}

\begin{remark}
Notations $(m_{\pm},n_{\pm})$ for the Jost functions used here are different from
notations $(m_{\pm},n_{\pm})$ used in \cite{Pelinovsky-Shimabukuro}, where an additional transformation
was used to generate $n_{\pm}$ (denoted by $p_{\pm}$ in \cite{Pelinovsky-Shimabukuro}).
This additional transformation is not necessary for our further work.
\end{remark}

\subsection{Transformation of the Jost functions for large $\lambda$}

Assume $v \in L^{\infty}(\R)$ and define the transformation matrix by
\begin{equation}\label{e def T hat}
  \widehat{T}(v;\lambda) :=
  \left(
    \begin{array}{cc}
      1 & 0 \\
      v & \lambda
    \end{array}
  \right).
\end{equation}
Let $\psi$ be a solution of the spectral problem (\ref{L-spectral}) and define $\widehat{\Psi}:= \widehat{T} \psi$.
Straightforward computations show that $\widehat{\Psi}$ satisfies the equivalent linear equation
\begin{equation}\label{e Lax1 Psi hat}
  \widehat{\Psi}_x=\widehat{\mathcal{L}}\widehat{\Psi},
\end{equation}
with new linear operator
\begin{equation}
\label{L-expansion2}
   \widehat{\mathcal{L}}=\widehat{Q}_1(u,v)+ \frac{1}{\lambda^2} \widehat{Q}_2(u,v)+ \frac{i}{4}\left(\lambda^2- \frac{1}{\lambda^2}\right)\sigma_3
\end{equation}
where
\begin{equation*}
  \widehat{Q}_1(u,v) =
  \left(
    \begin{array}{cc}
      \frac{i}{4}(|u|^2+|v|^2) & -\frac{i}{2}\overline{v} \\
      v_x+\frac{i}{2}|u|^2v+ \frac{i}{2}u & -\frac{i}{4}(|u|^2+|v|^2) \\
    \end{array}
  \right), \quad
  \widehat{Q}_2(u,v) = -\frac{i}{2}
  \left(
    \begin{array}{cc}
     \overline{u}v & -\overline{u} \\
      v+\overline{u}v^2 & -\overline{u}v \\
    \end{array}
  \right).
\end{equation*}

We introduce the same variable $z := \lambda^2$ and note that the second term in (\ref{L-expansion2})
is now bounded for $z \in B_{\infty}$. The normalized Jost functions associated to the spectral problem
(\ref{e Lax1 Psi}) denoted by $\{\widehat{m}_{\pm},\widehat{n}_{\pm}\}$ can be obtained
from the original Jost functions $\{ \varphi_{\pm},\psi_{\pm}\}$ by the transformation formulas:
\begin{equation}\label{e m n hat trafo}
  \widehat{m}_{\pm}(x;z)=\widehat{T}(v(x);\lambda)\varphi_{\pm}(x;\lambda),\qquad
  \widehat{n}_{\pm}(x;z)= \lambda^{-1} \widehat{T}(v(x);\lambda)\phi_{\pm}(x;\lambda),
\end{equation}
subject to the constant boundary conditions at infinity:
\begin{equation}\label{e asymptotics psi m n hat}
    \lim_{x\to\pm\infty} \widehat{m}_{\pm}(x;\lambda)=e_1\quad\text{ and } \quad\lim_{x\to\pm\infty} \widehat{n}_{\pm}(x;\lambda)=e_2.
\end{equation}
The transformed Jost functions are
solutions to the following Volterra integral equations:
\begin{subequations}\label{e volterra m n hat}
\begin{eqnarray}\label{e volterra m hat}
 &  \phantom{t} & \phantom{te}
        \widehat{m}_{\pm}(x;z) =  \\
 \nonumber
 & \phantom{t} & \phantom{text} \phantom{text} e_1 + \int_{\pm\infty}^x
        \left(
          \begin{array}{cc}
            1 & 0 \\
            0 & e^{-\frac{i}{2}(z-z^{-1})(x-y)} \\
          \end{array}
        \right)
        \left[\widehat{Q}_1(u(y),v(y)) + z^{-1} \widehat{Q}_2(u(y),v(y))\right] \widehat{m}_{\pm}(y;z)dy,\\
\label{e volterra n hat}
&  \phantom{t} & \phantom{te}
\widehat{n}_{\pm}(x;z) =  \\
 \nonumber
 & \phantom{t} & \phantom{text} \phantom{text} e_2 + \int_{\pm\infty}^x
        \left(
          \begin{array}{cc}
            e^{\frac{i}{2}(z-z^{-1})(x-y)} & 0 \\
            0 & 1\\
          \end{array}
        \right)
        \left[\widehat{Q}_1(u(y),v(y))+ z^{-1} \widehat{Q}_2(u(y),v(y))\right] \widehat{n}_{\pm}(y;z)dy.
    \end{eqnarray}
\end{subequations}
Again, we have an additional term $\frac{i}{2} z^{-1} (x-y)$
in the argument of the oscillatory kernel and the additional term $z^{-1} \widehat{Q}_2(u,v)$ under the integration sign.
However, both additional terms are bounded in $B_{\infty}$ where $|z| > 1$.
The following two lemmas contain results analogous to Lemmas \ref{l solvability m n} and \ref{l expansion m n}.

\begin{lem}\label{l solvability m n hat}
  Let $(u,v)\in L^1(\R)\cap L^{\infty}(\R)$ and $v_x\in L^1(\R)$.
  For every $z \in \R \setminus [-1,1]$, there exist unique solutions
  $\widehat{m}_{\pm}(\cdot;z)\in L^{\infty}(\R)$ and $\widehat{n}_{\pm}(\cdot;z)\in L^{\infty}(\R)$
  satisfying the integral equations (\ref{e volterra m n hat}). For every $x\in\R$,
  $\widehat{m}_{\pm}(x,\cdot)$ and $\widehat{n}_{\mp}(x,\cdot)$ are continued analytically in
  $\C^{\pm} \cap B_{\infty}$. There exist a positive constant $C$ such that
  \begin{equation}\label{e global bound m n hat}
    \| \widehat{m}_{\pm}(\cdot;z)\|_{L^{\infty}} +
    \| \widehat{n}_{\mp}(\cdot;z)\|_{L^{\infty}}\leq C, \quad z \in \C^{\pm} \cap B_{\infty}.
  \end{equation}
\end{lem}

\begin{lem}\label{l expansion m n hat}
Under the conditions of Lemma \ref{l solvability m n hat}, for every $x\in\R$
the normalized Jost functions $\widehat{m}_{\pm}$ and $\widehat{n}_{\pm}$ satisfy
  the following limits as $\Imag(z) \to \infty$ along a contour in the domains of their analyticity:
  \begin{equation}\label{e expansion m n hat}
    \lim_{|z| \to \infty} \frac{\widehat{m}_{\pm}(x;z)}{\widehat{m}_{\pm}^{\infty}(x)} = e_1, \quad
    \lim_{|z| \to \infty} \frac{\widehat{n}_{\pm}(x;z)}{\widehat{n}_{\pm}^{\infty}(x)} = e_2,
   \end{equation}
where
  \begin{equation*}
    \widehat{m}_{\pm}^{\infty}(x)= e^{\frac{i}{4}\int_{\pm\infty}^{x}(|u|^2+|v|^2) dy}, \quad
    \widehat{n}_{\pm}^{\infty}(x)= e^{-\frac{i}{4}\int_{\pm\infty}^{x}(|u|^2+|v|^2)dy}.
  \end{equation*}
If in addition $v \in C^1(\mathbb{R})$, then
\begin{subequations} \label{e expansion m n-longer hat}
\begin{eqnarray} \label{e expansion m-longer hat}
    \lim_{|z| \to \infty} z \left[ \frac{\widehat{m}_{\pm}(x;z)}{\widehat{m}_{\pm}^{\infty}(x)} - e_1 \right] & = &
      \left(
        \begin{array}{c}
          - \int_{\pm\infty}^x \left[\overline{v} (v_x+\frac{i}{2}|u|^2v+ \frac{i}{2}u)+\frac{i}{2} \overline{u}v\right]dy \\
          - 2 i v_x + |u|^2 v + u \\
        \end{array}
      \right),\\
\label{e expansion n-longer hat}
    \lim_{|z| \to \infty} z \left[ \frac{\widehat{n}_{\pm}(x;z)}{\widehat{n}_{\pm}^{\infty}(x)} - e_2 \right] & = &
      \left(
        \begin{array}{c}
          \overline{v}\\
          \int_{\pm\infty}^x \left[\overline{v}(v_x+\frac{i}{2}|u|^2 v + \frac{i}{2} u) + \frac{i}{2}\overline{u} v \right] dy
        \end{array}
      \right).
    \end{eqnarray}
\end{subequations}
\end{lem}

\subsection{Continuation of the transformed Jost functions across $\mathbb{S}^1$}

In Lemmas \ref{l solvability m n} and \ref{l solvability m n hat} we showed
the existence of the transformed Jost functions
$$
\{m_{\pm}(\cdot;z),n_{\pm}(\cdot;z)\}, \quad z \in B_0, \qquad
\mbox{\rm and} \qquad
\{\widehat{m}_{\pm}(\cdot;z),\widehat{n}_{\pm}(\cdot;z)\}, \quad z \in B_{\infty},
$$
respectively, where the partition (\ref{e def B_1}) is used. Because both sets of
the transformed Jost functions are connected to the set
$\{\varphi_{\pm},\phi_{\pm}\}$ of the original Jost functions by the transformation formulas
(\ref{e m n trafo}) and (\ref{e m n hat trafo}), respectively, we find
the following connection formulas for every $z \in \mathbb{S}^1$:
\begin{subequations}\label{e contuniation formulas}
\begin{eqnarray}
  \label{e contuniation formula m} m_{\pm}(x;z) &=&
  \left(
    \begin{array}{cc}
      1 & 0 \\
      u(x) - z^{-1} v(x) & z^{-1} \\
    \end{array}
  \right) \widehat{m}_{\pm}(x;z),
   \\
  \label{e contuniation formula n}n_{\pm}(x;z) &=&
  \left(
    \begin{array}{cc}
      z & 0 \\
      u(x) z - v(x) & 1 \\
    \end{array}
  \right) \widehat{n}_{\pm}(x;z),
\end{eqnarray}
\end{subequations}
or in the opposite direction,
\begin{subequations}\label{e contuniation formulas backward}
\begin{eqnarray}
  \label{e contuniation formula m hat}
  \widehat{m}_{\pm}(x;z) &=&
  \left(
    \begin{array}{cc}
      1 & 0 \\
      v(x)- z u(x) & z\\
    \end{array}
  \right) m_{\pm}(x;z),
  \\
  \label{e contuniation formula n hat}\widehat{n}_{\pm}(x;z) &=&
  \left(
    \begin{array}{cc}
      z^{-1} & 0 \\
      v(x) z^{-1} - u(x) & 1 \\
    \end{array}
  \right) n_{\pm}(x;z).
\end{eqnarray}
\end{subequations}

By Lemmas \ref{l solvability m n hat} and \ref{l expansion m n hat},
the right-hand sides of (\ref{e contuniation formula m}) and (\ref{e contuniation formula n}) yield
analytic continuations of $m_{\pm}(x;\cdot)$ and $n_{\mp}(x;\cdot)$ in $\C^{\pm} \cap B_{\infty}$ respectively
with the following limits as ${\rm Im}(z) \to \infty$ along a contour in the domains of their analyticity:
\begin{equation}
\label{limits-m-n}
  \lim_{|z| \to \infty} \frac{m_{\pm}(x;z)}{\widehat{m}_{\pm}^{\infty}(x)} = e_1 + u(x) e_2, \quad
    \lim_{|z| \to \infty} \frac{n_{\pm}(x;z)}{\widehat{n}_{\pm}^{\infty}(x)} = \bar{v}(x) e_1 + (1 + u(x) \bar{v}(x)) e_2.
\end{equation}
Analogously, by Lemmas \ref{l solvability m n} and \ref{l expansion m n}, the right-hand sides of
(\ref{e contuniation formula m   hat}) and (\ref{e contuniation formula n   hat})
yield analytic continuations of $\widehat{m}_{\pm}(x;\cdot)$ and $\widehat{n}_{\mp}(x;\cdot)$ in $\C^{\pm} \cap B_0$
respectively with the following limits as ${\rm Im}(z) \to 0$ along a contour in the domains of their analyticity:
\begin{equation}
\label{limits-m-n-hats}
  \lim_{z \to 0} \frac{\widehat{m}_{\pm}(x;z)}{m_{\pm}^{\infty}(x)} = e_1 + v(x) e_2, \quad
    \lim_{z \to 0} \frac{\widehat{n}_{\pm}(x;z)}{n_{\pm}^{\infty}(x)} = \bar{u}(x) e_1 + (1 + \bar{u}(x) v(x)) e_2.
\end{equation}
By Lemmas \ref{l solvability m n}, \ref{l expansion m n}, \ref{l solvability m n hat}, \ref{l expansion m n hat}, and the continuation formulas
(\ref{e contuniation formulas}), (\ref{e contuniation formulas backward}),
we obtain the following result.

\begin{lem}\label{l analytic continuation}
  Let $(u,v)\in L^1(\R)\cap L^{\infty}(\R)$ and $(u_x,v_x) \in L^1(\R)$.. For every $x\in\R$ the Jost functions
   defined by the integral equations (\ref{e volterra m n}) and (\ref{e volterra m n hat})
   can be continued such that $m_{\pm}(x;\cdot)$, $n_{\mp}(x;\cdot)$, $\widehat{m}_{\pm}(x;\cdot)$, and
  $\widehat{n}_{\mp}(x;\cdot)$ are analytic in $\C^{\pm}$ and continuous in $\C^{\pm} \cup\R$
  with bounded limits as $z \to 0$ and $|z| \to \infty$ given by (\ref{e expansion m n}), (\ref{e expansion m n hat}),
  (\ref{limits-m-n}), (\ref{limits-m-n-hats}).
\end{lem}

\section{Scattering coefficients}
\label{sec-coef}

In order to define the scattering coefficients between the transformed Jost functions $\{ m_{\pm},n_{\pm}\}$
and $\{\widehat{m}_{\pm},\widehat{n}_{\pm}\}$,  we go back to the original Jost functions  $\{ \varphi_{\pm},\phi_{\pm}\}$.
For every $\lambda\in(\R\cup i\R)\setminus\{0\}$, we define the standard form of the scattering relation by
\begin{eqnarray}\label{e scattering relation}
\left( \begin{array}{c} \varphi_{-}(x;\lambda)e^{i x(\lambda^2-\lambda^{-2})/4} \\
\phi_{-}(x;\lambda)e^{-i x(\lambda^2-\lambda^{-2})/4} \end{array} \right)
= \left( \begin{array}{cc} \alpha(\lambda) & \beta(\lambda) \\
\gamma(\lambda) & \delta(\lambda) \end{array} \right)
\left( \begin{array}{c} \varphi_{+}(x;\lambda)e^{i x(\lambda^2-\lambda^{-2})/4} \\
\phi_{+}(x;\lambda)e^{-i x(\lambda^2-\lambda^{-2})/4} \end{array} \right).
    \end{eqnarray}
Since the operator $L$ in (\ref{e def L}) admits the symmetry
\begin{equation*}
  \overline{\phi_{\pm}(x;\lambda)}=\pm
  \left(
    \begin{array}{cc}
      0 & -1 \\
      1 & 0 \\
    \end{array}
  \right) \varphi_{\pm}\left(x;\overline{\lambda}\right),
\end{equation*}
we obtain
\begin{equation}
\label{alpha-beta-relations}
  \gamma(\lambda)=-\overline{\beta(\overline{\lambda})}, \quad \delta(\lambda)=\overline{\alpha(\overline{\lambda})},
  \quad \lambda\in(\R\cup i\R)\setminus\{0\}.
\end{equation}

Since the matrix operator $L$ in (\ref{e def L}) has zero trace,
the Wronskian determinant $W$ of any two solutions to the spectral problem (\ref{L-spectral})
for any $\lambda \in \mathbb{C}$ is independent of $x$. By
computing the Wronskian determinants of the solutions $\{ \varphi_-,\phi_+ \}$
and $\{ \varphi_+,\varphi_-\}$ as $x \to +\infty$ and
using the scattering relation (\ref{e scattering relation}) and
the asymptotic behavior of the Jost functions $\{ \varphi_{\pm},\psi_{\pm}\}$, we obtain
\begin{equation}\label{e wronskian for alpha and beta}
\left\{ \begin{array}{l}
    \alpha(\lambda) = W\left(\varphi_-(x;\lambda)e^{i x(\lambda^2-\lambda^{-2})/4}, \phi_+(x;\lambda) e^{-i x(\lambda^2-\lambda^{-2})/4}\right),\\
    \beta(\lambda) = W\left(\varphi_{+}(x;\lambda)e^{i x(\lambda^2-\lambda^{-2})/4},\varphi_{-}(x;\lambda)e^{i x(\lambda^2-\lambda^{-2})/4}\right).
  \end{array} \right.
\end{equation}
It follows from the asymptotic behavior of $\{ \varphi_-,\phi_-\}$ as $x \to -\infty$
that $W(\varphi_-,\phi_-) = 1$. Substituting (\ref{e scattering relation})
and using the the asymptotic behavior of $\{ \varphi_+,\phi_+ \}$ as $x \to +\infty$
yield the following constraint on the scattering data:
\begin{equation}
\label{data-constraint}
\alpha(\lambda) \delta(\lambda) - \beta(\lambda) \gamma(\lambda) = 1, \quad \lambda \in (\R\cup i\R)\setminus\{0\}.
\end{equation}
In view of the constraints (\ref{alpha-beta-relations}), the constraint (\ref{data-constraint})
can be written as
\begin{equation}
\label{data-constraint-new}
\alpha(\lambda) \overline{\alpha(\overline{\lambda})} +
\beta(\lambda) \overline{\beta(\overline{\lambda})} = 1, \quad \lambda \in (\R\cup i\R)\setminus\{0\}.
\end{equation}

By using the transformation formulas (\ref{e m n trafo}) we can rewrite
the scattering relation (\ref{e scattering relation}) in terms of
the transformed Jost functions $\{m_{\pm},n_{\pm}\}$.
In particular, we apply $T(u;\lambda)$ to the first equation in (\ref{e scattering relation})
and $\lambda T(u;\lambda)$ to the second equation in (\ref{e scattering relation}), so that
we obtain for $z \in \mathbb{R} \backslash \{0\}$,
\begin{eqnarray}
\label{e scattering relation m n}
\left( \begin{array}{c}
         m_{-}(x;z)e^{i x(z-z^{-1})/4} \\ n_{-}(x;z)e^{-i x(z-z^{-1})/4}
                  \end{array} \right) = \left( \begin{array}{cc} a(z) & b_+(z) \\
         -\overline{b_-(z)} & \overline{a(z)}  \end{array} \right)
         \left( \begin{array}{c}  m_{+}(x;z) e^{i x(z-z^{-1})/4} \\ n_{+}(x;z)e^{-i x(z-z^{-1})/4}\end{array} \right),
    \end{eqnarray}
where we have recalled $z = \lambda^2$ and defined the scattering coefficients:
\begin{equation}\label{e def a b+ b-}
  a(z) := \alpha(\lambda),\quad
  b_+(z) := \lambda^{-1} \beta(\lambda),\quad
  b_-(z) := \lambda \beta(\lambda),\quad z \in \mathbb{R} \backslash\{0\}.
\end{equation}
Since $m_{\pm}(x;z)$ and $n_{\pm}(x;z)$ depend on $z = \lambda^2$,
we deduce that $\alpha$ is even in $\lambda$ and $\beta$ is odd in $\lambda$ for $\lambda \in (\R\cup i\R)\setminus\{0\}$.
The latter condition yields $\overline{\lambda} \beta(\overline{\lambda}) = \lambda \beta(\lambda)$,
which has been used already in the expression (\ref{e def a b+ b-}) for $b_-(z)$.
Thanks to the relation (\ref{data-constraint-new}),
we have the following constraints
\begin{equation}
\label{scattering-relation-modified}
\left\{ \begin{array}{l}
|\alpha(\lambda)|^2 + |\beta(\lambda)|^2 = 1, \quad \lambda \in \mathbb{R} \backslash \{0\}, \\
|\alpha(\lambda)|^2 - |\beta(\lambda)|^2 = 1, \quad \lambda \in i\mathbb{R} \backslash \{0\}.
\end{array} \right.
\end{equation}

Since the matrix operator $\mathcal{L}$ in (\ref{L-expansion1}) has zero trace,
the Wronskian determinant $W$ of any two solutions to the spectral problem (\ref{e Lax1 Psi})
is also independent of $x$. As a result, by
computing the Wronskian determinant as $x \to +\infty$ and
using the asymptotic behavior of the Jost functions $\{ m_{\pm},n_{\pm}\}$,
we obtain from the scattering relation (\ref{e scattering relation m n})
for $z \in \mathbb{R} \backslash \{0\}$:
\begin{equation}\label{e wronskian for a b+ b-}
\left\{ \begin{array}{l} a(z) = W\left(m_{-}(x;z)e^{i x(z-z^{-1})/4},n_{+}(x;z)e^{-i x(z-z^{-1})/4}\right), \\
    b_+(z) = W\left(m_{+}(x;z)e^{i x(z-z^{-1})/4}, m_{-}(x;z)e^{i x(z-z^{-1})/4}\right),\\
    \overline{b_-(z)} = W\left(n_+(x;z)e^{-i x(z-z^{-1})/4}, n_-(x;z)e^{-i x(z-z^{-1})/4}\right),
\end{array} \right.
\end{equation}
in accordance with the representation (\ref{e wronskian for alpha and beta}).

Analogously, by using the transformation formulas (\ref{e m n hat trafo}) we can rewrite
the scattering relation (\ref{e scattering relation}) in terms of
the transformed Jost functions $\{\widehat{m}_{\pm},\widehat{n}_{\pm}\}$.
In particular, we apply $\widehat{T}(u;\lambda)$ to the first equation in (\ref{e scattering relation})
and $\lambda^{-1}\widehat{T}(u;\lambda)$ to the second equation in (\ref{e scattering relation}),
so that we obtain for $z\in\R  \backslash \{0\}$,
\begin{eqnarray}
\label{e scattering relation m n hat}
\left( \begin{array}{c}
         \widehat{m}_{-}(x;z)e^{i x(z-z^{-1})/4} \\
             \widehat{n}_{-}(x;z)e^{-i x(z-z^{-1})/4} \end{array} \right)
         = \left( \begin{array}{cc} \widehat{a}(z) & \widehat{b}_+(z) \\ - \overline{\widehat{b}_-(z)} & \overline{\widehat{a}(z)} \end{array} \right)
         \left( \begin{array}{c} \widehat{m}_{+}(x;z)e^{i x(z-z^{-1})/4} \\ \widehat{n}_{+}(x;z)e^{-i x(z-z^{-1})/4} \end{array} \right),
    \end{eqnarray}
where we have recalled $z = \lambda^2$ and defined the scattering coefficients
\begin{equation}\label{e def a b+ b- hat}
  \widehat{a}(z):=\alpha(\lambda),\quad
  \widehat{b}_+(z):=\lambda \beta(\lambda),\quad
  \widehat{b}_-(z):= \lambda^{-1} \beta(\lambda),\quad z\in\R \backslash \{0\}.
\end{equation}
Since the matrix operator $\widehat{\mathcal{L}}$ in (\ref{L-expansion2}) has zero trace,
we obtain from the scattering relation (\ref{e scattering relation m n hat}) for $z \in \mathbb{R} \backslash \{0\}$:
\begin{equation}\label{e wronskian for a b+ b- hat}
\left\{ \begin{array}{l} \widehat{a}(z) = W\left(\widehat{m}_{-}(x;z)e^{i x(z-z^{-1})/4},\widehat{n}_{+}(x;z)e^{-i x(z-z^{-1})/4}\right), \\
    \widehat{b}_+(z) = W\left(\widehat{m}_{+}(x;z)e^{i x(z-z^{-1})/4}, \widehat{m}_{-}(x;z)e^{i x(z-z^{-1})/4}\right),\\
    \overline{\widehat{b}_-(z)} = W\left(\widehat{n}_+(x;z)e^{-i x(z-z^{-1})/4}, \widehat{n}_-(x;z)e^{-i x(z-z^{-1})/4}\right),
\end{array} \right.
\end{equation}
in accordance with the representation (\ref{e wronskian for alpha and beta}).

It follows from (\ref{e def a b+ b-}) and (\ref{e def a b+ b- hat}) that the two sets of scattering data
are actually related by
\begin{equation}
\label{a-b-relations}
\widehat{a}(z) = a(z), \quad \widehat{b}_+(z) = b_-(z), \quad \widehat{b}_-(z) = b_+(z), \quad z \in \mathbb{R} \backslash \{0\}.
\end{equation}
These relations are in agreement with the continuation formulas (\ref{e contuniation formulas}) and (\ref{e contuniation formulas backward}).
By using the representations (\ref{e wronskian for a b+ b-}) and (\ref{e wronskian for a b+ b- hat}), as well as
Lemma \ref{l expansion m n}, \ref{l expansion m n hat}, and \ref{l analytic continuation}, we obtain the following.

\begin{lem}\label{l scattering coeff}
  Let $(u,v)\in L^1(\R)\cap L^{\infty}(\R)$ and $(u_x,v_x) \in L^1(\R)$. Then, $a = \widehat{a}$ is continued analytically into $\C^-$ with the
  following limits in $\C^-$:
        \begin{equation}
        \label{a-infty}
        \lim_{z \to 0} a(z) = e^{-\frac{i}{4} \int_{\mathbb{R}}(|u|^2 + |v|^2) dy} =:  a_0
        \end{equation}
        and
        \begin{equation}
        \label{a-hat-infty}
        \lim_{|z| \to\infty} a(z) = e^{\frac{i}{4}\int_{\mathbb{R}}(|u|^2 + |v|^2) dy} =:  a_{\infty}.
        \end{equation}
On the other hand, $b_{\pm} = \widehat{b}_{\pm}$ are not continued analytically beyond the real line and
satisfy the following limits on $\mathbb{R}$:
  \begin{equation}
  \label{b-infty}
  \lim_{z \to 0} b_{\pm}(z) = \lim_{|z| \to \infty} b_{\pm}(z) = 0.
  \end{equation}
\end{lem}

To simplify the inverse scattering transform, we consider the case of no eigenvalues or resonances
in the spectral problem (\ref{L-spectral}) in the sense of the following definition.

\begin{mydef}
\label{def-eigenvalues}
We say that the potential $(u,v)$ admits an eigenvalue at $z_0 \in \mathbb{C}^-$ if $a(z_0) = 0$
and a resonance at $z_0 \in \R$ if $a(z_0) = 0$.
\end{mydef}

By taking the limit $x \to +\infty$ in the Volterra integral equations (\ref{e volterra m}) and (\ref{e volterra m hat})
for $m_-$ and $\widehat{m}_-$ respectively and comparing it with the first equations in the scattering
relations (\ref{e scattering relation m n}) and (\ref{e scattering relation m n hat}), we obtain
the following equivalent representations for $a = \widehat{a}$:
\begin{subequations}\label{a integral}
\begin{eqnarray}
\label{a-zero-integral}
& \phantom{text} & a(z) =  1 - \frac{i}{4} \int_{\mathbb{R}} \left[ (|u|^2 + |v|^2) m_-^{(1)} - 2 \bar{u} m_-^{(2)}
- 2z \bar{v} (u m_-^{(1)} - m_-^{(2)}) \right] dx, \;\; z \in B_0 \cap \C^-, \\
\label{a-infty-integral}
& \phantom{text} & a(z) =  1 + \frac{i}{4} \int_{\mathbb{R}} \left[ (|u|^2 + |v|^2) \widehat{m}_-^{(1)} - 2 \bar{v} \widehat{m}_-^{(2)}
- 2z^{-1} \bar{u} (v \widehat{m}_-^{(1)} - \widehat{m}_-^{(2)}) \right] dx,  z \in B_{\infty} \cap \C^-,
\end{eqnarray}
\end{subequations}
where the superscripts denote components of the Jost functions. If $(u,v) \in H^{1,1}(\R)$ are defined
in the ball of radius $\delta$ for some $\delta \in (0,1)$, then constants $C$ in (\ref{e global bound m n})
and (\ref{e global bound m n hat}) are independent of $\delta$. Then, it follows from (\ref{a integral})
that if $\delta$ is sufficiently small, then the integrals can be made as small as needed for every $z \in \C^- \cup \R$.
This implies the following.

\begin{lem}\label{lem-small-data}
  Let $(u,v)\in L^1(\R)\cap L^{\infty}(\R)$ and $(u_x,v_x) \in L^1(\R)$ be sufficiently small. Then $(u,v)$ does not admit
  eigenvalues or resonances in the sense of Definition \ref{def-eigenvalues}.
\end{lem}

\begin{remark}
The result of Lemma \ref{lem-small-data} was first obtained in Theorem 6.1 in \cite{Pel-survey}.
No transformation of the spectral problem (\ref{L-spectral}) was employed in \cite{Pel-survey}.
Transformations similar to those we are using here were employed later in \cite{Pelinovsky-Shimabukuro}
in the context of the derivative NLS equation.
\end{remark}

\begin{remark}
The result of Lemma \ref{lem-small-data} is useful for the study of long-range scattering
from small initial data. Eigenvalues can always be included
by using B\"{a}cklund transformation for the MTM system \cite{CPS,PSS}.
Resonances are structurally unstable and can be removed by perturbations of initial data \cite{Beals1984,PSS}.
\end{remark}

\section{Riemann--Hilbert problems}
\label{sec-RH}

We will derive two Riemann--Hilbert problems. The first problem is formulated for the transformed Jost functions
$\{m_{\pm},n_{\pm}\}$, whereas the second problem is formulated for the transformed Jost functions
$\{\widehat{m}_{\pm},\widehat{n}_{\pm}\}$. Thanks to the asymptotic representations
(\ref{e expansion m n-longer}) and (\ref{e expansion m n-longer hat}),
the first problem is useful for reconstruction
of the component $u$ as $z \to 0$, whereas the second problem is useful for reconstruction of the component $v$
as $|z| \to \infty$, both components satisfy the MTM system (\ref{e mtm}).
This pioneering idea has first appeared on a formal level in \cite{V}.
The following assumption is used to simplify solutions to the Riemann--Hilbert problems.

\begin{assumption}
\label{assumption-a}
Assume that the scattering coefficient $a$ admits no zeros in $\mathbb{C}^-\cup \mathbb{R}$.
\end{assumption}

Assumption \ref{assumption-a} corresponds to the initial data $(u_0,v_0)$ which admit
no eigenvalues or resonances in the sense of Definition \ref{def-eigenvalues}. By Lemma \ref{lem-small-data},
the assumption is satisfied if the $H^{1,1}(\R)$ norm on the initial data
is sufficiently small. Since $a$ is continued analytically into $\mathbb{C}^-$ by Lemma \ref{l scattering coeff}
with nonzero limits (\ref{a-infty}) and (\ref{a-hat-infty}), zeros of $a$ lie in a compact set.
Therefore, if $a$ admits no zeros in $\mathbb{C}^- \cup \mathbb{R}$ by Assumption \ref{assumption-a},
then there is $A > 0$ such that $|a(z)| \geq A$ for every $z \in \mathbb{R}$.

\subsection{Riemann-Hilbert problem for the potential $u$}

The asymptotic limit (\ref{limits-m-n})
presents a challenge to use $\{m_{\pm},n_{\pm}\}$ for reconstruction
of $(u,v)$ as $|z| \to \infty$. On the other hand, the reconstruction
formula for $(u,v)$ in terms of $\{m_{\pm},n_{\pm}\}$ is available from the asymptotic limit
(\ref{e expansion m n-longer}) as $z \to 0$. In order to avoid this complication, we use the inversion transformation
$\omega = 1/z$, which maps $0$ to $\infty$ and vice versa. The analyticity regions
swap under the inversion transformation so that $\{ m_-,n_+\}$ become analytic in
$\C^+$ for $\omega$ and $\{ m_+, n_-\}$ become analytic in $\C^-$ for $\omega$.

Let us define matrices $P_{\pm}(x;\omega) \in\C^{2\times 2}$ for every $x \in \mathbb{R}$ and $\omega \in \R$ by
\begin{equation}
  P_+(x;\omega) := \left[\frac{m_{-}(x;\omega^{-1})}{a(\omega^{-1})}, n_{+}(x;\omega^{-1})\right],\quad
  P_-(x;\omega) :=\left[m_{+}(x;\omega^{-1}),\frac{n_{-}(x;\omega^{-1})} {\overline{a(\omega^{-1})}} \right],
\end{equation}
and two reflection coefficients
\begin{equation}\label{e def r+-}
  r_{\pm}(\omega)= \frac{b_{\pm}(\omega^{-1})}{a(\omega^{-1})},\quad \omega \in \R,
\end{equation}
The scattering relation (\ref{e scattering relation m n}) can be rewritten as
the following jump condition for  the Riemann--Hilbert problem:
\begin{equation*}
   P_+(x;\omega) = P_-(x;\omega)
  \left[
    \begin{array}{cc}
      1+r_+(\omega) \overline{r_-(\omega)}  & \overline{r_-(\omega)}e^{-\frac{i}{2} (\omega-\omega^{-1})x} \\
      r_+(\omega) e^{\frac{i}{2}(\omega-\omega^{-1}) x} & 1 \\
    \end{array}
  \right]
\end{equation*}

If the scattering coefficient $a$ satisfies Assumption \ref{assumption-a}, then
$P_{\pm}(x;\cdot)$ for every $x \in \mathbb{R}$ are continued analytically in $\C^{\pm}$ by
Lemmas \ref{l analytic continuation} and \ref{l scattering coeff}.
We denote these continuations by the same letters. Asymptotic limits
(\ref{e expansion m n}) and (\ref{a-infty}) yield
the following behavior of $P_{\pm}(x;\omega)$ for large $|\omega|$
in the domains of their analyticity:
\begin{equation*}
  P_{\pm}(x;\omega)\to
  \left[
    \begin{array}{cc}
      m_{+}^{\infty}(x) & 0 \\
      0 & n_{+}^{\infty}(x) \\
    \end{array}
  \right] =: P^{\infty}(x)\quad\text{as }|\omega|\to\infty.
\end{equation*}
Since we prefer to work with $x$-independent boundary conditions, we normalize
the boundary conditions by defining
\begin{equation}\label{e def M}
  M_{\pm}(x;\omega) := \left[ P^{\infty}(x) \right]^{-1} P_{\pm}(x;\omega), \quad \omega \in \C^{\pm}.
\end{equation}
The following Riemann-Hilbert problem is formulated for the function $M(x;\cdot)$.
\begin{samepage}
\begin{framed}
\begin{rhp}\label{rhp M stat}
For each $x\in\R$, find a $2\times 2$-matrix valued function $M(x;\cdot)$ such that
\begin{enumerate}
  \item $M(x;\cdot)$ is piecewise analytic in $\C\setminus\R$ with continuous boundary values
  \begin{equation*}
  M_{\pm}(x;\omega)=\lim_{\varepsilon\downarrow 0} M(x;\omega \pm i\varepsilon), \quad z\in\mathbb{R}.
  \end{equation*}
  \item $M(x;\omega) \to I$ as $|\omega| \to \infty$.
  \item The boundary values $M_{\pm}(x;\cdot)$ on $\R$ satisfy the jump relation
      \begin{equation*}
        M_+(x;\omega) - M_-(x;\omega) = M_-(x;\omega) R(x;\omega), \quad \omega \in \R,
      \end{equation*}
      where
      \begin{equation*}
        R(x;\omega):=
        \left[
         \begin{array}{cc}
           r_+(\omega) \overline{r_-(\omega)}  & \overline{r_-(\omega)} e^{-\frac{i}{2} (\omega-\omega^{-1})x} \\
          r_+(\omega)e^{\frac{i}{2}(\omega-\omega^{-1}) x} & 0 \\
         \end{array}
        \right].
        \end{equation*}
\end{enumerate}
\end{rhp}
\end{framed}
\end{samepage}

It follows from the asymptotic limits (\ref{e expansion m n-longer}) and the normalization
(\ref{e def M}) that the components $(u,v)$ of the MTM system (\ref{e mtm}) are related to
the solution of the Riemann--Hilbert problem \ref{rhp M stat} by using the following reconstruction formulas:
\begin{equation}\label{e rec u 1}
\left[ 2 i u'(x) + u(x)|v(x)|^2 + v(x)\right] e^{\frac{i}{2} \int_{x}^{+\infty} (|u|^2+|v|^2) dy}
= \lim_{|\omega| \to \infty} \omega [M(x;\omega)]_{21}
\end{equation}
and
\begin{equation}\label{e rec u 2}
 \overline{u}(x)e^{-\frac{i}{2} \int_{x}^{+\infty} (|u|^2+|v|^2) dy} =\lim_{|\omega| \to \infty} \omega [M(x;\omega)]_{12},
\end{equation}
where the subscript denotes the element of the $2 \times 2$ matrix $M$.

\begin{remark}
The gauge factors in (\ref{e rec u 1})--(\ref{e rec u 2}) appear because of the normalization (\ref{e def M})
and the asymptotic limits (\ref{e expansion m n-longer}). A different approach was utilized in \cite{Perry1,Perry}
to avoid these gauge factors. The inverse scattering transform was developed to a different spectral problem,
which was obtained from the Kaup--Newell spectral problem after a gauge transformation.
\end{remark}

\subsection{Riemann-Hilbert problem for the potential $v$}

Let us define matrices $\widehat{P}_{\pm}(x;z) \in\C^{2\times 2}$ for every $x \in \R$ and $z \in \R$ by
\begin{equation}
  \widehat{P}_+(x;z):= \left[ \widehat{m}_{+}(x;z), \frac{\widehat{n}_{-}(x;z)} {\overline{\widehat{a}(z)}}\right],\quad
  \widehat{P}_-(x;z):= \left[\frac{\widehat{m}_{-}(x;z)}{\widehat{a}(z)}, \widehat{n}_{+}(x;z)\right],
\end{equation}
and two reflection coefficients by
\begin{equation}\label{e def r+- hat}
  \widehat{r}_{\pm}(z) = \frac{\widehat{b}_{\pm}(z)}{\widehat{a}(z)} = \frac{b_{\mp}(z)}{a(z)}, \quad z\in \R,
\end{equation}
where the relations (\ref{a-b-relations}) have been used.
The scattering relation (\ref{e scattering relation m n hat}) can be rewritten as
the following jump condition for  the Riemann--Hilbert problem:
\begin{equation*}
   \widehat{P}_+(x;z) = \widehat{P}_-(x;z)
  \left[
    \begin{array}{cc}
      1 & -\overline{\widehat{r}_-(z)} e^{\frac{i}{2} (z-z^{-1})x} \\
      -\widehat{r}_+(z)e^{-\frac{i}{2}(z-z^{-1}) x} & 1+ \widehat{r}_+(z) \overline{\widehat{r}_-(z)} \\
    \end{array}
  \right]
\end{equation*}

If the scattering coefficient $a$ satisfies Assumption \ref{assumption-a}, then
$\widehat{P}_{\pm}(x;\cdot)$  for every $x \in \mathbb{R}$ are continued analytically in $\C^{\pm}$ by
Lemmas \ref{l analytic continuation} and \ref{l scattering coeff}.
We denote these continuations by the same letters. Asymptotic limits (\ref{e expansion m n hat})
and (\ref{a-hat-infty}) yield the following behavior of $\widehat{P}(x;z)$ for large $|z|$
in the domains of their analyticity:
\begin{equation*}
  \widehat{P}_{\pm}(x;z) \to
  \left[
    \begin{array}{cc}
      \widehat{m}_{+}^{\infty}(x) & 0 \\
      0 & \widehat{n}_{+}^{\infty}(x) \\
    \end{array}
  \right] =: \widehat{P}^{\infty}(x),\quad\text{as }|z|\to\infty.
\end{equation*}
In order to normalize the boundary conditions, we define
\begin{equation}\label{e def M hat}
  \widehat{M}_{\pm}(x;z):= \left[ \widehat{P}^{\infty}(x) \right]^{-1} \widehat{P}_{\pm}(x;z), \quad z\in\C^{\pm}.
\end{equation}
The following Riemann-Hilbert problem is formulated for the function $\widehat{M}(x;\cdot)$.
\begin{samepage}
\begin{framed}
\begin{rhp}\label{rhp M stat hat}
For each $x\in\R$, find a $2\times 2$-matrix valued function $\widehat{M}(x;\cdot)$ such that
\begin{enumerate}
  \item $\widehat{M}(x;\cdot)$ is piecewise analytic in $\C\setminus\R$ with continuous boundary values
  \begin{equation*}
  \widehat{M}_{\pm}(x;z)=\lim_{\varepsilon\downarrow 0} \widehat{M}(x;z\pm i\varepsilon), \quad z\in\mathbb{R}.
  \end{equation*}
  \item $\widehat{M}(x;z) \to I$ as $|z|\to\infty$.
  \item The boundary values $\widehat{M}_{\pm}(x;\cdot)$ on $\R$ satisfy the jump relation
      \begin{equation*}
        \widehat{M}_+(x;z) - \widehat{M}_-(x;z) = \widehat{M}_-(x;z) \widehat{R}(x;z),
        \end{equation*}
        where
        \begin{equation*}
        \widehat{R}(x;z):=
        \left[
          \begin{array}{cc}
             0  & -\overline{\widehat{r}_-(z)}e^{\frac{i}{2} (z-z^{-1})x} \\
             -\widehat{r}_+(z)e^{-\frac{i}{2}(z-z^{-1}) x} & \widehat{r}_+(z) \overline{\widehat{r}_-(z)} \\
          \end{array}
        \right].
      \end{equation*}
\end{enumerate}
\end{rhp}
\end{framed}
\end{samepage}

It follows  from the asymptotic limit (\ref{e expansion m n-longer hat}) and the normalization
(\ref{e def M hat}) that the components $(u,v)$ of the MTM system (\ref{e mtm}) can be recovered
from the solution of the Riemann--Hilbert problem \ref{rhp M stat hat} by using the following reconstruction formulas:
\begin{equation}\label{e rec v 1}
  \left[ - 2i v'(x) + |u(x)|^2 v(x) + u(x)\right] e^{-\frac{i}{2} \int_{x}^{+\infty} (|u|^2+|v|^2) dy} =
  \lim_{|z|\to\infty} z \left[\widehat{M}(x;z)\right]_{21}
\end{equation}
and
\begin{equation}\label{e rec v 2}
 \overline{v}(x)e^{\frac{i}{2} \int_{x}^{+\infty} (|u|^2+|v|^2) dy} =\lim_{|z|\to\infty} z \left[\widehat{M}(x;z)\right]_{12},
\end{equation}
where the subscript denotes the element of the $2 \times 2$ matrix $M$.

Let us now outline the reconstruction procedure for $(u,v)$ as a solution of the MTM system (\ref{e mtm})
in the inverse scattering transform. If the right-hand sides of (\ref{e rec u 2}) and (\ref{e rec v 2}) are controlled
in the space $H^1(\R) \cap L^{2,1}(\R)$, then $(\tilde{u},\tilde{v}) \in H^1(\R) \cap L^{2,1}(\R)$, where
$$
\tilde{u}(x) = u(x)e^{\frac{i}{2} \int_{x}^{+\infty} (|u|^2+|v|^2) dy}, \quad
\tilde{v}(x) = v(x)e^{-\frac{i}{2} \int_{x}^{+\infty} (|u|^2+|v|^2) dy}.
$$
Since $|\tilde{u}(x)| = |u(x)|$ and $|\tilde{v}(x)| = |v(x)|$, the gauge factors can be immediately inverted,
and since $H^1(\mathbb{R})$ is continuously embedded into $L^p(\mathbb{R})$ for any $p \geq 2$, we then have
$(u,v)  \in H^1(\R) \cap L^{2,1}(\R)$. If the right-hand sides of (\ref{e rec u 1}) and (\ref{e rec v 1}) are
also controlled in $H^1(\R) \cap L^{2,1}(\R)$, then similar arguments give $(u',v') \in H^1(\R) \cap L^{2,1}(\R)$,
that is, $(u,v) \in H^2(\R) \cap H^{1,1}(\R)$, in agreement with the function space used for direct scattering transform.

\begin{remark}
It follows from the limit (\ref{b-infty}) that $R(x;0) = \widehat{R}(x;0) = 0$
implying $M_+(x;0) = M_-(x;0)$ and $\widehat{M}_+(x;0) = \widehat{M}_-(x;0)$.
More precisely, using (\ref{limits-m-n}), (\ref{limits-m-n-hats}), (\ref{a-infty}),
(\ref{a-hat-infty}), and  $\omega = z^{-1}$ we can derive
\begin{equation*}
M(x;0) =  \left[
      \begin{array}{cc}
        m_+^{\infty}(x) & 0 \\
        0 & n_+^{\infty}(x) \\
      \end{array}
    \right]^{-1}
    \left[
      \begin{array}{cc}
        1 & \overline{v}(x) \\
        u(x) & 1+u(x)\overline{v}(x) \\
      \end{array}
    \right]
    \left[
      \begin{array}{cc}
        \widehat{m}_+^{\infty}(x) & 0 \\
        0 & \widehat{n}_+^{\infty}(x) \\
      \end{array}
    \right]
\end{equation*}
and
\begin{equation*}
        \widehat{M}(x;0) =  \left[
      \begin{array}{cc}
        \widehat{m}_+^{\infty}(x) & 0 \\
        0 & \widehat{n}_+^{\infty}(x) \\
      \end{array}
    \right]^{-1}
    \left[
      \begin{array}{cc}
        1 & \overline{u}(x) \\
        v(x) & 1+\overline{u}(x) v(x) \\
      \end{array}
    \right]
    \left[
      \begin{array}{cc}
        m_+^{\infty}(x) & 0 \\
        0 & n_+^{\infty}(x) \\
      \end{array}
    \right].
\end{equation*}
In particular, the following holds:
\begin{equation*}
  [M(x;0)]_{11} = \frac{\widehat{m}_+^{\infty}(x)}{m_+^{\infty}(x)} = e^{-\frac{i}{2} \int_{x}^{+\infty} (|u|^2+|v|^2) dy}, \qquad
  [\widehat{M}(x;0)]_{11} = \frac{m_+^{\infty}(x)}{\widehat{m}_+^{\infty}(x)} = e^{\frac{i}{2} \int_{x}^{+\infty} (|u|^2+|v|^2) dy}.
\end{equation*}
In these formulas, we regain the same exponential factors as those in the reconstruction formulas
(\ref{e rec u 2}) and (\ref{e rec v 2}). Hence, by substitution we obtain
the following two decoupled reconstruction formulas:
\begin{equation}\label{e rec alternative}
 u(x) = [M(x;0)]_{11} \overline{\lim_{|\omega|\to\infty} \omega [M(x;\omega)]_{12}}, \qquad
 v(x) = [\widehat{M}(x;0)]_{11} \overline{\lim_{|z|\to\infty} z [\widehat{M}(x;z)]_{12}}.
\end{equation}
Whereas equations (\ref{e rec u 1}), (\ref{e rec u 2}), (\ref{e rec v 1}) and (\ref{e rec v 2}) are suitable for studying the inverse map of the scattering transformation in the sense of Theorem \ref{theorem-main}, the equivalent formulas (\ref{e rec alternative}) are useful in the analysis of the asymptotic behavior of $u(x)$ and $v(x)$ as $|x| \to \infty$.
\end{remark}

\subsection{Estimates on the reflection coefficients}

In order to be able to solve the Riemann--Hilbert problems \ref{rhp M stat} and \ref{rhp M stat hat},
we need to derive estimates on the reflection coefficients $r_{\pm}$ and $\widehat{r}_{\pm}$
defined by (\ref{e def r+-}) and (\ref{e def r+- hat}). We start with the Jost functions.
In order to exclude ambiguity in notations, we write $m_{\pm}(x;z) \in H^1_z(\R)$ for the same purpose
as $m_{\pm}(x;\cdot) \in H^1(\R)$.

Thanks to the Fourier theory reviewed in Proposition 1 in \cite{Pelinovsky-Shimabukuro},
the Volterra integral equations (\ref{e volterra m n}) and (\ref{e volterra m n hat})
with the oscillation factors $e^{\frac{i}{2}(\omega^{-1} - \omega)}$ and
$e^{\frac{i}{2} (z - z^{-1}) x}$ are estimated respectively in the limits $|\omega| \to \infty$
and $|z| \to \infty$, where $\omega := z^{-1}$,  similarly to what was done in the proof of Lemma 3 in \cite{Pelinovsky-Shimabukuro}.
As a result, we obtain the following.

\begin{lem}\label{l remainder}
  Let $(u,v) \in H^{1,1}(\R)$. Then for every $x\in\R^{\pm}$, we have
  \begin{equation}\label{e remainder inverted}
  m_{\pm}(x;\omega^{-1}) - m_{\pm}^{\infty}(x) e_1 \in H^1_{\omega}(\R \setminus [-1,1]),\quad
  n_{\pm}(x;\omega^{-1}) - n_{\pm}^{\infty}(x) e_2 \in H^1_{\omega}(\R \setminus [-1,1]).
  \end{equation}
and
  \begin{equation}\label{e remainder hat}
      \widehat{m}_{\pm}(x;z)-\widehat{m}_{\pm} ^{\infty}(x)e_1 \in H^1_z(\R \setminus [-1,1]), \quad
      \widehat{n}_{\pm}(x;z)-\widehat{n}_{\pm}^{\infty}(x)e_2 \in H^1_z(\R\setminus[-1,1]).
  \end{equation}
If in addition $(u,v) \in H^2(\R)$, then
\begin{subequations} \label{e remainder extended inverted}
\begin{eqnarray} \label{e remainder extended m inverted}
&& \omega \left[ \frac{m_{\pm}(x;\omega^{-1})}{m_{\pm}^{\infty}(x)} - e_1 \right] -
      \left(
        \begin{array}{c}
          -\int_{\pm\infty}^x \left[ \overline{u} (u_x-\frac{i}{2}u|v|^2- \frac{i}{2}v)-\frac{i}{2} u\overline{v}\right]dy \\
          2i u_x + u |v|^2 + v
        \end{array}
      \right) \in L^2_{\omega}(\R \setminus [-1,1]), \\
\label{e remainder extended n inverted}
&&   \omega \left[ \frac{n_{\pm}(x;\omega^{-1})}{n_{\pm}^{\infty}(x)} - e_2 \right] -
  \left(
        \begin{array}{c}
          \overline{u}\\
          \int_{\pm\infty}^x \left[\overline{u}(u_x-\frac{i}{2}u|v|^2- \frac{i}{2}v)-\frac{i}{2}u\overline{v}\right] dy \\
        \end{array}
      \right) \in L^2_{\omega}(\R \setminus [-1,1]).
    \end{eqnarray}
\end{subequations}
and
\begin{subequations} \label{e remainder extended hat}
\begin{eqnarray} \label{e remainder extended hat m}
&&  z \left[ \frac{\widehat{m}_{\pm}(x;z)}{\widehat{m}_{\pm}^{\infty}(x)} - e_1 \right] -
      \left(
        \begin{array}{c}
          - \int_{\pm\infty}^x \left[\overline{v} (v_x+\frac{i}{2}|u|^2v+ \frac{i}{2}u)+\frac{i}{2} \overline{u}v\right]dy \\
          - 2 i v_x + |u|^2 v + u \\
        \end{array}
      \right) \in L^2_z(\R \setminus [-1,1]),\\
      \label{e remainder extended hat n}
&& z \left[ \frac{\widehat{n}_{\pm}(x;z)}{\widehat{n}_{\pm}^{\infty}(x)} - e_2 \right] -
  \left(
        \begin{array}{c}
          \overline{v}\\
          \int_{\pm\infty}^x \left[\overline{v}(v_x+\frac{i}{2}|u|^2 v + \frac{i}{2} u) + \frac{i}{2}\overline{u} v \right] dy
        \end{array}
      \right) \in L^2_z(\R\setminus[-1,1]).
    \end{eqnarray}
\end{subequations}
\end{lem}

The following lemma transfers the estimates of Lemma \ref{l remainder} to the scattering coefficients $a$ and $b_{\pm}$
by using the same analysis as in the proof of Lemma 4 in \cite{Pelinovsky-Shimabukuro}.

\begin{lem}\label{l scattering regularity}
Let $(u,v)\in H^{1,1}(\R)$. Then,
        \begin{equation}\label{e rgularity scattering coeff}
          a(\omega^{-1}) - a_0, \;\; b_+(\omega^{-1}), \;\; b_-(\omega^{-1}) \in H^1_{\omega}(\R\setminus[-1,1]),
        \end{equation}
        and
        \begin{equation}\label{e rgularity scattering coeff hat}
          a(z) - a_{\infty}, \;\; b_+(z), \;\; b_-(z) \in H^1_z(\R\setminus[-1,1]).
        \end{equation}
If in addition $(u,v) \in H^2(\R)$, then
        \begin{equation}\label{e rgularity scattering coeff more}
b_+(\omega^{-1}), \;\; b_-(\omega^{-1}) \in L^{2,1}_{\omega}(\R\setminus[-1,1]),
        \end{equation}
        and
        \begin{equation}\label{e rgularity scattering coeff more hat}
b_+(z), \;\; b_-(z) \in L^{2,1}_z(\R\setminus[-1,1]).
        \end{equation}
\end{lem}

The following lemma transfers the estimates of Lemma \ref{l scattering regularity} to
the reflection coefficients $r_{\pm}$ and $\widehat{r}_{\pm}$. We give an elementary
proof of this result since it is based on new computations compared to \cite{Pelinovsky-Shimabukuro}.

\begin{lem}
\label{lemma-r-coefficients}
Assume $(u,v) \in X_{(u,v)}$, where $X_{(u,v)}$ is given by (\ref{X-2-1}),
and $a$ satisfies Assumption \ref{assumption-a}.
Then $(r_+,r_-) \in X_{(r_+,r_-)}$, where $X_{(r_+,r_-)}$ is given by (\ref{X-1-2}).
\end{lem}

\begin{proof}
Under the conditions of the lemma, it follows from Lemma \ref{l scattering regularity} and
from the definitions (\ref{e def r+-}) and (\ref{e def r+- hat}) that
$$
r_{\pm}(\omega) \in \dot{H}^1_{\omega}(\R\setminus[-1,1]) \cap \dot{L}^{2,1}_{\omega}(\R\setminus[-1,1])
$$
and
$$
\hat{r}_{\pm}(\omega) \in \dot{H}^1_{z}(\R\setminus[-1,1]) \cap \dot{L}^{2,1}_{z}(\R\setminus[-1,1]).
$$
It also follows from (\ref{e def r+-}) and (\ref{e def r+- hat}) that $r_{\pm}(\omega)=\widehat{r}_{\mp}(\omega^{-1})$.

If $f(x) \in \dot{L}^{2,1}_x(1,\infty)$ and $\tilde{f}(y) := f(y^{-1})$, then $\tilde{f}(y) \in \dot{L}^{2,-2}_y(0,1)$,
which follows by the chain rule:
\begin{equation*}
  \int_{1}^{\infty}x^2|f(x)|^2dx= \int_{0}^{1} y^{-4} |\tilde{f}(y)|^2dy.
\end{equation*}
Since $\dot{L}^{2,1}(1,\infty)$ is continuously embedded into $\dot{L}^{2,-2}(1,\infty)$
and $\dot{L}^{2,-2}(0,1)$ is continuously embedded into $\dot{L}^{2,1}(0,1)$,
we verify that $r_{\pm}(z) \in \dot{L}^{2,1}_z(\R) \cap \dot{L}^{2,-2}_z(\R)$
and $\hat{r}_{\pm}(\omega) \in \dot{L}^{2,1}_{\omega}(\R) \cap \dot{L}^{2,-2}_{\omega}(\R)$.

Finally, if $f(x) \in \dot{H}^1_x(1,\infty)$ and $\tilde{f}(y) := f(y^{-1})$, then $\tilde{f}(y) \in \dot{H}^{1,1}_y(0,1)$,
which follows by the chain rule $f'(x)=-x^{-2}\tilde{f}'(x^{-1})$ and
\begin{equation*}
  \int_{1}^{\infty}|f'(x)|^2dx= \int_{0}^{1}y^2|\tilde{f}'(y)|^2dy.
\end{equation*}
Combing all requirements together, we obtain the space $X_{(r_+,r_-)}$ both
for $(r_+,r_-)$ in $z$ and for $(\hat{r}_+,\hat{r}_-)$ in $\omega$,
where $X_{(r_+,r_-)}$ is given by (\ref{X-1-2}).
\end{proof}

\begin{remark}
\label{remark-about-r}
It follows from the relations (\ref{e def a b+ b-}) and (\ref{e def a b+ b- hat})  that $r_+(\omega) = \omega r_-(\omega)$ and
$\widehat{r}_+(z) = z \widehat{r}_-(z)$. Then, it follows from Lemma \ref{lemma-r-coefficients}
and the chain rule that
$$
\mbox{\rm if} \;\; r_+,\hat{r}_+ \in \dot{H}^1(\R \setminus [-1,1]) \cap \dot{L}^{2,1}(\R), \;\;
\mbox{\rm then} \;\;  r_-, \hat{r}_- \in \dot{H}^{1,1}(\R \setminus [-1,1]) \cap \dot{L}^{2,2}(\R)
$$
and
$$
\mbox{\rm if} \;\; r_-, \hat{r}_- \in \dot{H}^{1,1}([-1,1]) \cap \dot{L}^{2,-2}(\R) \;\;
\mbox{\rm then} \;\; r_+, \hat{r}_+ \in \dot{H}^1([-1,1]) \cap \dot{L}^{2,-3}(\R).
$$
Therefore, we have $r_+,\hat{r}_+ \in \dot{H}^1(\R) \cap \dot{L}^{2,1}(\R) \cap \dot{L}^{2,-3}(\R)$
and $r_-,\hat{r}_- \in \dot{H}^{1,1}(\R) \cap \dot{L}^{2,2}(\R) \cap \dot{L}^{2,-2}(\R)$.
\end{remark}

\begin{remark}
It may appear strange for the first glance that the direct and inverse scattering transforms for the MTM system (\ref{e mtm})
connect potentials $(u,v) \in X_{(u,v)}$ and reflection coefficients $(r_+,r_-) \in X_{(r_+,r_-)}$ in different spaces,
whereas the Fourier transform provides an isomorphism in the space $H^1(\R) \cap L^{2,1}(\R)$.
However, the appearance of $X_{(u,v)}$ spaces for the potential $(u,v)$ is not surprising
due to the transformation of the linear operator $L$ to the equivalent forms (\ref{L-expansion1}) and (\ref{L-expansion2}).
The condition $(u,v) \in X_{(u,v)}$ ensures that $(Q_{1,2}, \hat{Q}_{1,2}) \in H^1(\R) \cap L^{2,1}(\R)$, hence, the direct
and inverse scattering transform for the MTM system (\ref{e mtm}) provides a transformation
between $(Q_{1,2}, \hat{Q}_{1,2}) \in H^1(\R) \cap L^{2,1}(\R)$ and $(r_+,r_-) \in X_{(r_+,r_-)}$, which is a natural transformation
under the Fourier transform with oscillatory phase $e^{ix(\omega - \omega^{-1})}$.
This allows us to avoid reproducing the Fourier analysis anew and to apply
all the technical results from \cite{Pelinovsky-Shimabukuro} without any changes, as these results generalize
the classical results of Deift \& Zhou \cite{D-Z-1,Z} obtained for the cubic NLS equation.
\end{remark}

\subsection{Solvability of the Riemann--Hilbert problems}

Let us define the reflection coefficient
\begin{equation}
\label{reflection-coeff}
r(\lambda) := \frac{\beta(\lambda)}{\alpha(\lambda)}, \quad \lambda \in \R \cup (i \R) \backslash \{0\}.
\end{equation}
Recall the relations (\ref{e def a b+ b-}), (\ref{e def a b+ b- hat}), (\ref{e def r+-}),
and (\ref{e def r+- hat}) which yield
\begin{equation}
\label{r-relation-1}
\lambda^{-1} r(\lambda) = r_+(\omega) = \omega r_-(\omega), \quad \omega \in \R \backslash \{0\}.
\end{equation}
and
\begin{equation}
\label{r-relation-2}
\lambda r(\lambda) = \widehat{r}_+(z) = z \widehat{r}_-(z), \quad z \in \R \backslash \{0\}.
\end{equation}
Also recall that $z = \lambda^2$ and $\omega = \lambda^{-2}$.
By extending the proof of Propositions 2 and 3 in \cite{Pelinovsky-Shimabukuro},
we obtain the following.

\begin{lem}
If $(r_+,r_-) \in X_{(r_+,r_-)}$, then
\begin{equation}
\label{r-1}
r(\lambda) \in L^{2,1}_{\omega}(\R) \cap L^{\infty}_{\omega}(\R), \quad
r(\lambda) \in L^{2,1}_{z}(\R) \cap L^{\infty}_{z}(\R),
\end{equation}
and
\begin{equation}
\label{r-3}
\lambda^{-1} r_+(\omega) \in L^{\infty}_{\omega}(\R), \quad
\lambda \hat{r}_+(z) \in L^{\infty}_z(\R).
\end{equation}
\end{lem}

\begin{proof}
Let us prove the embeddings in $L^2_z(\R)$ space. The proof of the embeddings in $L^2_{\omega}(\R)$ space
is analogous. Relation (\ref{r-relation-2}) implies $|r(\lambda)|^2 = |\hat{r}_+(z)| |\hat{r}_-(z)|$ and
$$
r(\lambda) = \left\{ \begin{array}{l} \lambda^{-1} \hat{r}_+(z), \quad |z| \geq 1, \\
\lambda \hat{r}_-(z), \quad \quad |z| \leq 1. \end{array} \right.
$$
Since $\hat{r}_+,\hat{r}_- \in L^{2,1}(\R)$,
Cauchy--Schwarz inequality implies $r(\lambda) \in L^{2,1}_z(\R)$.
Since $\hat{r}_+ \in H^1(\R)$ by Remark \ref{remark-about-r},
$r(\lambda) \in L^{\infty}_z(\R \setminus [-1,1])$. In order to prove
that $r(\lambda) \in L^{\infty}_z([-1,1])$, we will show
that $\lambda \hat{r}_-(z) \in L^{\infty}_z([-1,1])$. This follows from
the representation
$$
z \hat{r}_-(z)^2 = \int_0^z \left[ \hat{r}_-^2(z) + 2 z \hat{r}_-(z) \hat{r}_-'(z) \right] dz
$$
and the Cauchy--Schwarz inequality, since $\hat{r}_- \in \dot{H}^{1,1}(\R) \cap L^2(\R)$.
Similarly, $\lambda \hat{r}_+(z) \in L^{\infty}(\R)$ since $\hat{r}_+ \in H^1(\R) \cap L^{2,1}(\R)$.
\end{proof}

\begin{remark}
By using the relations (\ref{scattering-relation-modified}), we obtain another constraint on $r(\lambda)$:
\begin{equation}
\label{r-4}
1 - |r(\lambda)|^2 = \frac{1}{|\alpha(\lambda)|^2} \geq c_0^2 > 0, \quad \lambda \in i\mathbb{R},
\end{equation}
where $c_0 := \sup_{\lambda \in i \R} |\alpha(\lambda)| < \infty$, which exists thanks to Lemma \ref{l scattering coeff}.
\end{remark}

Under Assumption \ref{assumption-a} as well as the constraints (\ref{r-1}) and (\ref{r-4}), the jump matrices in
the Riemann--Hilbert problems \ref{rhp M stat} and \ref{rhp M stat hat}
satisfy the same estimates as in Proposition 5 in \cite{Pelinovsky-Shimabukuro}.
Hence these Riemann--Hilbert problems can be solved and estimated
with the same technique as in the proofs of Lemmas 7, 8, and 9 in \cite{Pelinovsky-Shimabukuro}.
The following summarizes this result.

\begin{lem}
\label{lemma-solution}
Under Assumption \ref{assumption-a}, for every $r(\lambda) \in L^2_{\omega}(\R) \cap L^{\infty}_{\omega}(\R)$
satisfying (\ref{r-4}), there exists a unique solution of the Riemann--Hilbert problem \ref{rhp M stat}
satisfying for every $x \in \R$:
\begin{equation}
\| M_{\pm}(x;\omega) - I \|_{L^2_{\omega}} \leq C \| r(\lambda) \|_{L^2_{\omega}},
\end{equation}
where the positive constant $C$ only depends on $\| r(\lambda) \|_{L^{\infty}_{\omega}}$. Similarly,
under Assumption \ref{assumption-a}, for every $r(\lambda) \in L^2_{z}(\R) \cap L^{\infty}_{z}(\R)$
satisfying (\ref{r-4}), there exists a unique solution of the Riemann--Hilbert problem \ref{rhp M stat hat}
satisfying for every $x \in \R$:
\begin{equation}
\| \widehat{M}_{\pm}(x;z) - I \|_{L^2_z} \leq \widehat{C} \| r(\lambda) \|_{L^2_z}
\end{equation}
where the positive constant $\widehat{C}$ only depends on $\| r(\lambda) \|_{L^{\infty}_z}$.
\end{lem}

The potentials $u$ and $v$ are recovered respectively from $M$ and $\widehat{M}$ by means of the reconstruction formulas
(\ref{e rec u 2}) and (\ref{e rec v 2}), whereas the derivatives of the potentials $u'$ and $v'$ are recovered
from the reconstruction formulas (\ref{e rec u 1}) and (\ref{e rec v 1}). At the first order in terms of the scattering coefficient
(see, e.g., \cite{Beals1984}), we have to analyze the integrals like
\begin{equation}\label{e BealsCoifman approx}
  \lim_{|\omega| \to \infty} \omega [M(x;\omega)]_{12}
   \sim \frac{i}{2\pi }\int_{\R} \overline{r_-(\omega)} e^{-\frac{i}{2} (\omega-\omega^{-1})x} d \omega
\end{equation}
in the space $H^1_x(\R) \cap L^{2,1}_x(\R)$. In order to control the remainder term of
the representation (\ref{e BealsCoifman approx}) in $H^1_x(\R) \cap L^{2,1}_x(\R)$,
we need to generalize Proposition 7 in \cite{Pelinovsky-Shimabukuro} for the case of the oscillatory factor
\begin{equation*}
  \Theta(s)=\frac12 \left(s-\frac{1}{s}\right).
\end{equation*}
The following lemma presents this generalization in the function space
$$
X_0 := H^1(\mathbb{R} \backslash [-1,1]) \cap \dot{H}^{1,1}([-1,1]) \cap \dot{L}^{2,-1}([-1,1]).
$$
The proof of this lemma is a non-trivial generalization of analysis of the Fourier integrals.

\begin{lem}\label{p estimate cauchy opearator}
There is a positive constant $C$ such that for all $x_0\in\R_+$ and all $f\in X_0$, we have
\begin{equation}\label{e technical estimate1}
\sup_{x\in(x_0,\infty)}\|\langle x\rangle\mathcal{P}^{\pm}[f(\diamond) e^{\mp ix\Theta(\diamond)}]\|_{L^2(\R)}\leq C \|f\|_{X_0}
\end{equation}
where $\langle x\rangle:=(1+x^2)^{1/2}$ and the Cauchy projection operators are explicitly given by
\begin{equation*}
  \mathcal{P}^{\pm}[f(\diamond)](z):=\lim_{\eps\downarrow 0}\frac{1}{2\pi i}
  \int_{\R}\frac{f(s)}{s-(z\pm i\eps)}ds,\quad z\in\R.
\end{equation*}
In addition, if $f\in X_0\cap \dot{L}^{2,-1}(\R)$, then
\begin{equation}\label{e technical estimate2}
\sup_{x\in\R}\|\mathcal{P}^{\pm}[f(\diamond) e^{\mp ix\Theta(\diamond)}]\|_{L^{\infty}(\R)}\leq
C \left( \|f\|_{X_0}+\|f\|_{\dot{L}^{2,-1}(\R)}\right).
\end{equation}
Furthermore, if $f\in L^{2,1}(\R) \cap \dot{L}^{2,-1}(\R)$, then
\begin{equation}\label{e technical estimate3}
\sup_{x\in\R}\|\mathcal{P}^{\pm} [(\diamond-\diamond^{-1})f(\diamond) e^{\mp ix\Theta(\diamond)}]\|_{L^2(\R)}\leq
C \left( \|f \|_{L^{2,1}(\R)} + \|f\|_{\dot{L}^{2,-1}(\R)} \right).
\end{equation}
\end{lem}

\begin{proof}
Consider the decomposition
$$
f(s) e^{\mp ix\Theta(s)} = f(s) e^{\mp ix\Theta(s)} \chi_{\R_-}(s) +
f(s) e^{\mp ix\Theta(s)} \chi_{\R_+}(s),
$$
where $\chi_S$ is a characteristic function on the set $S \subset \R$.
Thanks to the linearity of $\mathcal{P}^{\pm}$, it is sufficient to consider separately the functions $f$ that vanish either
on $\R_+$ or on $\R_-$. In the following we give an estimate for
$\mathcal{P}^{+}[f(\diamond) e^{- ix\Theta(\diamond)} \chi_{\R_+}(\diamond)]$.
The other case is handled analogously.

Fix $x > 0$ and consider the following decomposition:
\begin{equation}\label{e decomp of r}
  f(s)e^{-ix\Theta(s)} \chi_{\R_+}(s) = h_{I}(x,s)+h_{II}(x,s),
\end{equation}
with
\begin{equation*}
  h_{I}(x,s)=e^{-ix\Theta(s)}\frac{1}{2\pi} \int_{x/4}^{\infty}e^{ik(s-s^{-1})}\mathfrak{a}[f](k)dk
\end{equation*}
and
\begin{equation*}
  h_{II}(x,s)=e^{-i\frac{x}{4}(s-s^{-1})}\frac{1}{2\pi} \int_{-\infty}^{x/4}e^{i(k-\frac{x}{4})(s-s^{-1})}\mathfrak{a}[f](k)dk,
\end{equation*}
where
\begin{equation}\label{e def frac A}
  \mathfrak{a}[f](k):=\int_{0}^{\infty} e^{-ik(s-s^{-1})} \frac{1+s^2}{s^2} f(s) ds.
\end{equation}
The following change of coordinates
\begin{equation*}
  y(s)=s-s^{-1},\qquad s(y)=\frac{y}{2}+\sqrt{1+\frac{y^2}{4}}, \qquad
  s'(y)= \frac{1}{2}+\frac{y}{4} \left(\sqrt{1+\frac{y^2}{4}}\right)^{-1} = \frac{s(y)^2}{1+s(y)^2}
\end{equation*}
shows that $\mathfrak{a}[f](k)=\mathfrak{F}[\tilde{f}](k)$, where the function $\tilde{f}$ is given by
\begin{equation*}
  \tilde{f}(y)=f(s(y)),\quad y\in\R
\end{equation*}
and $\mathfrak{F}$ denotes the Fourier transform
\begin{equation*}
  \mathfrak{F}[\tilde{f}](k)=\int_{-\infty}^{\infty} e^{-iky} \tilde{f}(y) dy.
\end{equation*}
We obtain
\begin{equation*}
  \|\tilde{f}\|_{L^2(\R)}^2=\int_{\R} |f(s(y))|^2 dy =  \int_{0}^{\infty} \frac{1+s^2}{s^2} |f(s)|^2 ds \leq \|f\|^2_{X_0}
\end{equation*}
and
\begin{eqnarray*}
  \|\tilde{f}'\|_{L^2(\R)}^2 =\int_{\R} \left(\frac{s(y)^2}{1+s(y)^2}\right)^2|f'(s(y))|^2 dy =  \int_{0}^{\infty} \frac{s^2}{1+s^2} |f'(s)|^2 ds \leq \|f\|^2_{X_0}.
\end{eqnarray*}
It follows that $\tilde{f}\in H^1(\R)$ and thus by Fourier theory $\mathfrak{a}[f](k)\in L_k^{2,1}(\R)$. Using the inverse Fourier transform
\begin{equation*}
  \mathfrak{F}^{-1}[g](y)= \frac{1}{2\pi} \int_{\R}e^{iyk}g(k)dk,
\end{equation*}
we find for $s>0$:
\begin{equation}\label{e rep f(s) ito a(k)}
  f(s)=\tilde{f}(y(s)) \\
   = \mathfrak{F}^{-1}[\mathfrak{a}[f]](y(s))\\
   = \frac{1}{2\pi} \int_{\R}e^{ik(s-s^{-1})}\mathfrak{a}[f](k)dk.
\end{equation}

Addressing the decomposition (\ref{e decomp of r}), we obtain
for the functions $h_{I}$ thanks to $s'(y)<1$:
\begin{equation}\label{e L^2 norm hI}
  \|h_{I}(x,\cdot)\|^2_{L^2(\R_+)}\leq \left\|\frac{1}{2\pi} \int_{x/4}^{\infty}e^{iky}\mathfrak{a}[f](k)dk  \right\|_{L^{2}_y(\R)} = \int_{x/4}^{\infty}|\mathfrak{a}[f](k)|^2dk \leq \frac{C}{1+x^2} \|\mathfrak{a}[f]\|^2_{L^{2,1}(\R)}.
\end{equation}
On the other hand, the function $h_{II}(x,\cdot)$ is analytic in the domain $\{\Imag(s)<0\}$ and additionally for $s=-i\xi$ with $\xi\in\R_+$ we have
\begin{equation*}
  |h_{II}(x,s)|\leq C \|\mathfrak{a}[f]\|_{L^{2,1}(\R)} e^{-\frac{x}{4}(\xi+\xi^{-1})}.
\end{equation*}
Therefore, $\|h_{II}(x,\cdot)\|_{L^2(i\R_-)}$ is decaying exponentially as $x\to\infty$. Now we have
\begin{equation*}
  \|\mathcal{P}^{+}[f(\diamond) e^{- ix\Theta(\diamond)} \chi_{\R_+}(\diamond)]\|_{L^2(\R)}\leq  \|\mathcal{P}^+[h_{I}(x,\diamond) \chi_{\R_+}(\diamond)]\|_{L^2(\R)} +\|\mathcal{P}^+[h_{II}(x,\diamond) \chi_{\R_+}(\diamond)]\|_{L^2(\R)}
\end{equation*}
Since $\mathcal{P}^+$ is a bounded operator $L^2(\R_+)\to L^2(\R)$ it follows by (\ref{e L^2 norm hI}) that
$$
\|\mathcal{P}^+[h_{I}(x,\diamond) \chi_{\R_+}(\diamond)]\|_{L^2(\R)}\leq \|h_{I}(x,\cdot)\|^2_{L^2(\R_+)} \leq
C \langle x\rangle^{-1} \|f\|^2_{X_0}.
$$
Using a suitable path of integration and the analyticity of $h_{II}$ we find that
\begin{equation*}
    \mathcal{P}^+[h_{II}(x,\diamond)](z)= -\mathcal{P}_{i\R_-}[h_{II}(x,\diamond)](z),
\end{equation*}
where
\begin{equation*}
  \mathcal{P}_{i\R_-}[h](z):=\frac{1}{2\pi i}
  \int_{-\infty}^0\frac{h(is)}{is-z}ds,\quad z\in\R,
\end{equation*}
for a function $h:i\R_-\to\C$. Since $\mathcal{P}_{i\R_-}$ is a bounded operator $L^2(i\R_-)\to L^2(\R)$
(see, e.g., estimate (23.11) in \cite{Beals1988}) and because $\|h_{II}(x,\cdot)\|_{L^2(i\R_-)}$
is decaying exponentially as $x\to\infty$, the proof of the estimate (\ref{e technical estimate1}) is complete.

In order to prove the estimate (\ref{e technical estimate2}), we first note that for $z\leq 0$
\begin{eqnarray}
|\mathcal{P}^+[e^{-ix\Theta(\diamond)}f(\diamond) \chi_{\R_+}(\diamond)](z)| & \leq & \int_{0}^{\infty} \frac{|f(s)|}{s}ds
\nonumber \\
   &\leq & \left(\int_{0}^{1}\frac{|f(s)|^2}{s^2}ds\right)^{1/2}\!+ \! \left(\int_{1}^{\infty}\frac{1}{s^2}ds\right)^{1/2} \! \left(\int_{1}^{\infty}|f(s)|^2ds\right)^{1/2} \nonumber \\
   &\leq & C \left( \| f \|_{X_0} + \|f\|_{\dot{L}^{2,-1}(\R)} \right).\label{e estimate z<0}
\end{eqnarray}
Thus it remains to estimate $|\mathcal{P}^+[e^{-ix\Theta(\diamond)}f(\diamond) \chi_{\R_+}(\diamond)](z)|$ for $z>0$.
First, we will derive a bound for the special case $x=0$ and by (\ref{e shift property}) below we will see that the same bound holds for any $x\in\R$. Therfore, using (\ref{e rep f(s) ito a(k)}) we decompose
\begin{equation*}
  f(s)=h_+(s)+h_-(s),\qquad h_{\pm} (s):=\pm\frac{1}{2\pi} \int_{0}^{\pm\infty}e^{ik(s-s^{-1})}\mathfrak{a}[f](k)dk,
\end{equation*}
where $h_{\pm}$ has an analytic extension within the domain $\{s\in\C: \Real(s)>0,\pm\Imag(s)>0\}$ and for $\xi>0$ we have
\begin{equation}\label{e |h(s)|}
  |h_{\pm} (\pm i\xi)|\leq C \|e^{-k(\xi+\xi^{-1})}\|_{L^2(\R_+)} \|\mathfrak{a}[f]\|_{L^2_k(\R_{\pm})} =
  \frac{C}{\sqrt{2}} \sqrt{\frac{\xi}{1+\xi^2}} \|\mathfrak{a}[f]\|_{L^2_k(\R_{\pm})}.
\end{equation}
Using a residue calculation we obtain for $z>0$
\begin{eqnarray*}
  \mathcal{P}^{+}[f(\diamond) \chi_{\R_+}(\diamond)](z)
   &=& \lim_{\eps\downarrow 0}\frac{1}{2\pi i}
  \int_{0}^{\infty}\frac{h_+(s)+h_-(s)}{s-(z\pm i\eps)}ds  \\
   &=&  \mathcal{P}_{i\R_{+}}[h_+](z) -\mathcal{P}_{i\R_{-}}[h_-](z)+ h_{+}(z).
\end{eqnarray*}
Thanks to the bound (\ref{e |h(s)|}), the summands $\mathcal{P}_{i\R_{+}}[h_+](z)$ and $\mathcal{P}_{i\R_{-}}[h_-](z)$ are estimated in the following way,
\begin{eqnarray*}
  \sup_{z\in\R_+}|\mathcal{P}_{i\R_{\pm}}[h_{\pm}](z)| &\leq & \int_{0}^{\infty}\frac{|h_{\pm}(\pm i \xi)|}{\xi}d\xi \\
   &\leq&  C \int_{0}^{\infty}\frac{1}{\sqrt{\xi}\sqrt{1+\xi^2}}d\xi  \|\mathfrak{a}[f]\|_{L^2_k(\R_{\pm})} \\
   &\leq& C\|\mathfrak{a}[f]\|_{L^2_k(\R_{\pm})}.
\end{eqnarray*}
In addition, for $z>0$ we have $|h_{+}(z)|\leq \|\mathfrak{a}[f]\|_{L^1_k(\R_+)}$ so that
the triangle inequality implies:
\begin{equation}\label{e ghjk}
  \sup_{z\in\R_+}|\mathcal{P}^{+}[f(\diamond) \chi_{\R_+}(\diamond)](z)|\leq C\left( \|\mathfrak{a}[f]\|_{L^1(\R)}+\|\mathfrak{a}[f]\|_{L^2(\R)}\right).
\end{equation}
Now, let us reinsert the factor $e^{-ix\Theta(s)}$. From the definition of $\mathfrak{a}$ it follows that multiplication by $e^{-ix\Theta(s)}$ is equivalent of a shift of $\mathfrak{a}[f](k)$ in the $k$-variable,
\begin{equation}\label{e shift property}
  \mathfrak{a}[e^{-ix\Theta(\diamond)}f(\diamond)](k)= \mathfrak{a}[f(\diamond)]\left(k+\frac{x}{2}\right).
\end{equation}
Thus, the $L^1(\R)\cap L^2(\R)$-norm with respect to $k$ of $\mathfrak{a}[e^{-ix\Theta(\diamond)}f(\diamond)](k)$ does not depend on $x$. Therefore, (\ref{e ghjk}) yields
\begin{eqnarray}
\nonumber
  \sup_{z\in\R_+} |\mathcal{P}^{+}[e^{-ix\Theta(\diamond)}f(\diamond) \chi_{\R_+}(\diamond)](z)| & \leq &
  C\|\mathfrak{a}[e^{-ix\Theta(\diamond)} f(\diamond)]\|_{L^1(\R)\cap L^2(\R)} \\
  & = & C \|\mathfrak{a}[f]\|_{L^1(\R)\cap L^2(\R)}\leq C\|f\|_{X_0},
\end{eqnarray}
which, together with (\ref{e estimate z<0}), completes the proof of (\ref{e technical estimate2}).

Finally, the bound (\ref{e technical estimate3}) follows from $\|\mathcal{P}^{\pm}\|_{L^2\to L^2}=1$
and the fact that $(s-s^{-1})f(s)\in L_s^2(\R)$ if $f\in L^{2,1}(\R) \cap \dot{L}^{2,-1}(\R)$.
\end{proof}

The first term in (\ref{e BealsCoifman approx}) is estimated with a similar change of
coordinates $y := \omega-\omega^{-1}$ and further analysis in the proof of Lemma \ref{p estimate cauchy opearator}.
However, it is controlled $H^1_x(\R) \cap L^{2,1}_x(\R)$ if the scattering coefficient
$r_-$ is defined in $X_{(r_+,r_-)}$ according to the bound
\begin{equation}
\label{BC-bound}
\left| \int_{\R} \overline{r_-(\omega)} e^{-\frac{i}{2} (\omega-\omega^{-1})x} d \omega \right|_{H^1_x(\R) \cap L^{2,1}_x(\R)}
\leq C \| r_- \|_{X_{(r_+,r_-)}}.
\end{equation}
By using the estimate (\ref{BC-bound}) and the estimates of Lemma \ref{p estimate cauchy opearator},
we can proceed similarly to Lemmas 10, 11, and 12 in \cite{Pelinovsky-Shimabukuro}.
The following lemma summarize the estimates on the potential $(u,v)$ obtained from
the reconstruction formulas (\ref{e rec u 1})--(\ref{e rec u 2}) and (\ref{e rec v 1})--(\ref{e rec v 2}).

\begin{lem}
\label{lemma-recovery}
Under Assumption \ref{assumption-a}, for every $(r_+,r_-) \in X_{(r_+,r_-)}$
and $(\widehat{r}_+,\widehat{r}_-) \in X_{(r_+,r_-)}$,
the components $(u,v) \in X_{(u,v)}$ satisfy the bound
\begin{equation}
\| u \|_{H^2 \cap H^{1,1}} + \| v \|_{H^2 \cap H^{1,1}} \leq C
\left( \| r_+ \|_{X_{(r_+,r_-)}} + \| r_- \|_{X_{(r_+,r_-)}} + \| \hat{r}_+ \|_{X_{(r_+,r_-)}} + \| \hat{r}_- \|_{X_{(r_+,r_-)}} \right),
\end{equation}
where the positive constant $C$ depends on $\| r_{\pm} \|_{X_{(r_+,r_-)}}$ and $\| \hat{r}_{\pm} \|_{X_{(r_+,r_-)}}$.
\end{lem}

Lemma \ref{l scattering regularity} proves the first assertion of Theorem \ref{theorem-main}.
Lemma \ref{lemma-recovery} proves the second assertion of Theorem \ref{theorem-main} at $t = 0$.
It remains to prove the second assertion of Theorem \ref{theorem-main} for every $t \in \R$.

\subsection{Time evolution of the spectral data}

Thanks to the well-posedness result of Theorem \ref{theorem-well-posed} and standard
estimates in weighted $L^2$-based Sobolev spaces, there exists a global solution $(u,v) \in C(\R,X_{(u,v)})$ to the MTM system (\ref{e mtm})
for any initial data $(u,v)|_{t=0} = (u_0,v_0) \in X_{(u,v)}$. For this global solution,
the normalized Jost functions (\ref{norm-Jost}) can be extended for every $t \in \R$:
\begin{equation}
\label{norm-Jost-time}
\left\{ \begin{array}{l}
    \varphi_{\pm}(t,x;\lambda)=\psi^{(\pm)}_1(t,x;\lambda) e^{-i x(\lambda^2-\lambda^{-2})/4-i t(\lambda^2 + \lambda^{-2})/4},\\
    \phi_{\pm}(t,x;\lambda)=\psi^{(\pm)}_2(t,x;\lambda) e^{i x(\lambda^2-\lambda^{-2})/4 + i t(\lambda^2 + \lambda^{-2})/4}.
    \end{array} \right.
\end{equation}
where $(\varphi_{\pm},\phi_{\pm})$ still satisfy the same boundary conditions (\ref{e asymptotics psi}).
Introducing the scattering coefficients in the same way as in Section \ref{sec-coef}, we
obtain the time evolution of the scattering coefficients:
\begin{equation}
\label{coefficients-time}
\alpha(t,\lambda) = \alpha(0,\lambda), \quad
\beta(t,\lambda) = \beta(0,\lambda) e^{-i t(\lambda^2 + \lambda^{-2})/2}, \quad \lambda \in \R \cup (i \R) \backslash \{0\}.
\end{equation}
Transferring the scattering coefficients to the reflection coefficients
with the help of (\ref{e def a b+ b-}), (\ref{e def a b+ b- hat}), (\ref{e def r+-}), and (\ref{e def r+- hat})
yields the time evolution of the reflection coefficients:
\begin{equation}
\label{reflection-time}
r_{\pm}(t,\omega) = r_{\pm}(0,\omega) e^{-i t(\omega + \omega^{-1})/2}, \quad \omega \in \R \backslash \{0\}
\end{equation}
and
\begin{equation}
\label{reflection-time-hat}
\hat{r}_{\pm}(t,z) = \hat{r}_{\pm}(0,z) e^{-i t(z + z^{-1})/2}, \quad z \in \R \backslash \{0\}.
\end{equation}
It is now clear that if $r_{\pm}$ and $\hat{r}_{\pm}$ are in $X_{(r_+,r_-)}$ at the initial time $t = 0$,
then they remain in $X_{(r_+,r_-)}$ for every $t \in \R$. Thus, the recovery formulas of Lemma \ref{lemma-recovery}
for the global solution $(u,v) \in C(\R,X_{(u,v)})$ to the MTM system (\ref{e mtm})
hold for every $t \in \mathbb{R}$. This proves
the second assertion of Theorem \ref{theorem-main} for every $t \in \R$. Hence Theorem \ref{theorem-main} is proven.
\begin{remark}
  Adding the time dependence to the Riemann-Hilbert problem \ref{rhp M stat} we find the time-dependent jump relation $M_+(x,t;\omega) - M_-(x,t;\omega) = M_-(x,t;\omega) R(x,t;\omega)$,
  where
  \begin{equation*}
        R(x,t;\omega):=
        \left[
         \begin{array}{cc}
           r_+(\omega) \overline{r_-(\omega)}  & \overline{r_-(\omega)} e^{-\frac{i}{2} (\omega-\omega^{-1})x+\frac{i}{2} (\omega+\omega^{-1})t} \\
          r_+(\omega)e^{\frac{i}{2}(\omega-\omega^{-1}) x-\frac{i}{2} (\omega+\omega^{-1})t} & 0 \\
         \end{array}
        \right].
  \end{equation*}
The same phase function as in $R(x,t;\omega)$ appears in the inverse scattering theory
for the sine-Gordon equation. A Riemann-Hilbert problem
with such a phase function was studied in \cite{ChengVenakidesZhou1999},
where the long-time behavior of the sine-Gordon equation was analyzed.
\end{remark}

\begin{remark}
In the context of the MTM system (\ref{e mtm}), it is more natural to address global solutions
in weighted $H^1$ space such as $H^{1,1}(\R)$ and drop the requirement
$(u,v) \in H^2(\R)$. The scattering coefficients $r_{\pm}$ and $\hat{r}_{\pm}$ are then defined
in the space $X_0$. However, there are two obstacles to develop the inverse scattering transform
for such a larger class of initial data. First, the asymptotic limits (\ref{e expansion m-longer})
and (\ref{e expansion m-longer hat}) are not justified, therefore, the recovery formulas
(\ref{e rec u 1}) and (\ref{e rec v 1}) cannot be utilized. Second, without requirement
$r_{\pm}, \hat{r}_{\pm} \in L^{2,1}(\R)$, the time evolution (\ref{reflection-time})--(\ref{reflection-time-hat})
is not closed in $X_0$ since $r_-, \hat{r}_- \in L^{2,-2}(\R)$ cannot be verified.
In this sense, the space $X_{(u,v)}$ for $(u,v)$ and $X_{(r_+,r_-)}$ for $(r_+,r_-)$ and $(\widehat{r}_+,\widehat{r}_-)$
are optimal for the inverse scattering transform of the MTM system (\ref{e mtm}).
\end{remark}

\section{Conclusion}
\label{sec-conclusion}

We gave functional-analytical details on how the direct and inverse scattering transforms can be applied
to solve the initial-value problem for the MTM system in laboratory coordinates. We showed that initial data
$(u_0,v_0) \in X_{(u,v)}$ admitting no eigenvalues or resonances defines uniquely
the spectral data $(r_+,r_-)$ in $X_{(r_+,r_-)}$. With the time evolution added, the spectral
data $(r_+,r_-)$ remain in the space $X_{(r_+,r_-)}$ and determine uniquely the solution $(u,v)$ to the MTM system
(\ref{e mtm}) in the space $X_{(u,v)}$.

We conclude the paper with a list of open questions.

The long-range scattering of solutions to the MTM system (\ref{e mtm}) for small initial data
for which the assumption of no eigenvalues or resonances is justified can be considered
based on the inverse scattering transform presented here. This will be the subject of the forthcoming work,
where the long-range scattering results in \cite{CandyLindblad} obtained by regular functional-analytical methods
are to be improved.

The generalization of the inverse scattering transform in the case of eigenvalues is easy and can be performed
similarly to what was done for the derivative NLS equation in \cite{PSS}. However, it is not so easy to include
resonances and other spectral singularities in the inverse scattering transform. In particular, the case of
algebraic solitons \cite{KPR} corresponds to the spectral singularities of the scattering coefficients due to
slow decay of $(u,v)$ and analysis of this singular case is an open question.

Finally, another interesting question is to consider the inverse scattering transform for the initial data
decaying to constant (nonzero) boundary conditions. The MTM system (\ref{e mtm}) admits solitary waves
over the nonzero background \cite{BarGet} and analysis of spectral and orbital stability of such solitary waves
is at the infancy stage.

\end{document}